\documentclass[10pt]{amsart}

\title[Commutative character sheaves and geometric types]{Commutative character sheaves and geometric types for supercuspidal representations}
\author{Clifton Cunningham}
\address{Department of Mathematics and Statistics, University of Calgary, 2500 University Drive Northwest, Calgary, Alberta, Canada, {T2N~1N4}.}
\email{cunning@math.ucalgary.ca}
\author{David Roe}
\address{Department of Mathematics, Massachusetts Institute of Technology, 77 Massachusetts Ave, Cambridge, MA, United States, 02139.}
\email{roed@mit.edu}


\usepackage{amssymb}
\usepackage{amsrefs}
\usepackage{multirow}
\usepackage{mathrsfs}
\usepackage{enumitem}
\usepackage{hyperref}
\hypersetup{
  colorlinks   = true, 
  urlcolor     = blue, 
  linkcolor    = blue, 
  citecolor   = green 
}

\usepackage{tikz}
\usetikzlibrary{shapes,arrows,calc,matrix}
\usepackage{tikz-cd}

\usepackage{amsmath}
\usepackage{lipsum}
\usepackage{setspace}

\theoremstyle{plain}
      \newtheorem{theorem}{Theorem}[section]
      \newtheorem*{theorem*}{Theorem}
      \newtheorem{proposition}[theorem]{Proposition}
      \newtheorem{lemma}[theorem]{Lemma}
      \newtheorem{corollary}[theorem]{Corollary}

      \theoremstyle{definition}
      \newtheorem{definition}[theorem]{Definition}

      \newtheorem{remark}[theorem]{Remark}

\newcommand{\FFF}{{\mathbf{F}_3}}
\newcommand{\ZZ}{{\mathbb{Z}}}
\newcommand{\NN}{{\mathbb{N}}}
\newcommand{\CC}{{\mathbb{C}}}
\newcommand{\QQ}{{\mathbb{Q}}}

\newcommand{\EE}{\mathbb{\bar Q}_\ell}

\newcommand{\bFq}{\bar{k}}
\newcommand{\Fq}{k}

\newcommand{\EEx}{\EE^\times}

\newcommand{\Weil}[1]{\mathcal{W}_{#1}}

\newcommand{\Frob}[1]{\operatorname{Fr}_{#1}}
\DeclareMathOperator{\Aut}{Aut}
\DeclareMathOperator{\Hom}{Hom}

\DeclareMathOperator{\Gr}{Gr}

\DeclareMathOperator{\id}{id}
\DeclareMathOperator{\Ext}{Ext}
\DeclareMathOperator{\Hh}{H}

\DeclareMathOperator{\trace}{Tr}

\DeclareMathOperator{\image}{im}
\DeclareMathOperator{\Loc}{Loc}

\DeclareMathOperator{\SL}{SL}
\DeclareMathOperator{\PGL}{PGL}
\DeclareMathOperator{\PSL}{PSL}

\newcommand{\Spec}[1]{{\operatorname{Spec}(#1)}}

\newcommand{\der}{_{\operatorname{der}}}
\newcommand{\ab}{_{\operatorname{ab}}}

\newcommand{\ceq}{{\, :=\, }}
\newcommand{\tq}{{\ \vert\ }}
\newcommand{\iso}{{\ \cong\ }}
\newcommand{\trFrob}[1]{t_{#1}}
\DeclareMathOperator{\Tr}{Tr}
\newcommand{\TrFrob}[1]{\Tr_{#1}}

\newcommand{\cs}[1]{{\mathcal{#1}}}
\newcommand{\gcs}[1]{{\mathcal{\bar #1}}}

\newcommand{\IC}{\mathcal{IC}}
\newcommand{\CS}{{\mathcal{C\hskip-0.8pt S}}}
\newcommand{\CCS}{{\mathcal{C\hskip-.8pt C\hskip-0.8pt S}}}

\newcommand{\CSiso}[1]{\CS(#1)_{/\text{iso}}}

\newcommand{\CCSiso}[1]{\CCS(#1)_{/\text{iso}}}
\makeatletter
\newcommand{\labitem}[2]{
\def\@itemlabel{\textbf{#1}}
\item
\def\@currentlabel{#1}\label{#2}}
\makeatother

\newcommand{\bm}{\bar{m}}
\newcommand{\bG}{\bar{G}}

\newcommand{\tight}[3]{\hspace{-#1pt}{#2}\hspace{-#3pt}}

\newcommand{\LxL}{\text{$\gcs{L} \tight{0}{\boxtimes}{0} \gcs{L}$}}

\newcommand{\red}{^{\operatorname{red}}}
\newcommand{\Sp}{{\operatorname{Sp}}}

\newcommand{\oK}{{\,^\circ \hskip-1pt K}}
\newcommand{\orho}{{\,^\circ \hskip-1pt \rho}}

\newcommand\numberthis{\addtocounter{equation}{1}\tag{\theequation}}

\usepackage{todonotes}

\begin{document}

\begin{abstract}
We show that some types for supercuspidal representations of tamely ramified $p$-adic groups that appear in Jiu-Kang Yu's work are geometrizable.
To do so, we define a function-sheaf dictionary for one-dimen\-sion\-al characters of arbitrary smooth group schemes over finite fields.  
In previous work we considered the case of commutative smooth group schemes and found that the standard definition of character sheaves produced a dictionary with a nontrivial kernel.  
In this paper we give a modification of the category of character sheaves that remedies this defect, and is also extensible to non-commutative groups.  
We then use these \emph{commutative character sheaves} to geometrize the linear characters that appear in the types introduced by Jiu-Kang Yu, assuming that the character vanishes on a certain derived subgroup.
To define \emph{geometric types}, we combine commutative character sheaves with Gurevich and Hadani's geometrization of the Weil representation and Lusztig's character sheaves.
\end{abstract}

\maketitle

\tableofcontents

\section*{Introduction}

As proved by Ju-Lee Kim in \cite{kim:07a}, all irreducible supercuspidal representations of tamely ramified $p$-adic groups for sufficiently large $p$ can be built from ``data'' introduced by Jiu-Kang Yu in \cite{yu:01a}*{\S 15}.
While the type, in the sense of Bushnell \& Kutzko \cite{bushnell-kutzko:98a}, of a supercuspidal representation built from Yu data can be constructed directly from the datum, it is convenient to consider an intermediate object, introduced in \cite{yu:01a}*{Remark 15.4}, which we call a \emph{Yu type datum}. 
Yu type data are studied in \cite{Yu:models}, which concludes with the following observation.
\begin{quotation}
{\it Therefore, up to some linear characters, all the ingredient representations 
 are on groups of the form $\underline{H}(\mathcal{O})$, where $\underline{H}$ is a smooth group scheme over [a Henselian discrete valuation ring with finite residue field $\kappa$] $\mathcal{O}$, and the representations are inflated from $\underline{H}(\kappa)$. These results suggest that algebraic geometry and group schemes should play an
important role in the representation theory of $p$-adic groups.}
\end{quotation}

In this paper we follow the suggestion above by showing that certain Yu type data are geometrizable, in the following sense.
A Yu type datum determines a sequence of representations $\orho_i$ of compact $p$-adic groups $\oK^i$, for $i=0, \ldots, d$, such that $(\oK^d, \rho_d)$ is a type for a supercuspidal representation of a $p$-adic group.
Let $R$ be the ring of integers of a local field with finite residue field $k$.
The main result of \cite{Yu:models} shows how to find, for each $i=0, \ldots, d$, a smooth group scheme $\underline{G}^i$ over the ring $R$ with $\underline{G}^i(R)=\,^\circ K^i $.
Under certain assumptions on the Yu type datum, we show how each representation $\orho_i$ can be replaced by a pair $(\underline{G}^i, \cs{F}^i)$, where $\cs{F}^i$ is a formal ${\bar\QQ}$-linear conbination of conjugation-equivariant sheaf complexes on the Greenberg transform $G^i$ of $\underline{G}^i$.
Writing $\trFrob{\cs{F}^i}$ for the function on $G^i(k) = \underline{G}^i(R) = \,^\circ K^i$ obtained by evaluating the trace of the action of Frobenius on $\cs{F}^i$, we show in Theorem~\ref{thm:geotypes} that 
\begin{equation}\label{eqn:intro1}
\trFrob{\cs{F}^i} = \trace(\,^\circ\rho_i).
\end{equation}
By this theorem, then, we obtain geometric avatars for each type in a Yu datum:
\[
\begin{tikzcd}
(\oK^i,\orho_i) \arrow[bend left, dashed]{rrr}{\text{geometrization}} &&& \arrow[bend left, dashed]{lll}{\text{trace of Frob}}
(\underline{G}^i, \cs{F}^i).
\end{tikzcd}
\]
We refer to the pair $(\underline{G}^d, \cs{F}^d)$ as a \emph{geometric type}.

To prove Theorem~\ref{thm:geotypes}, we must find a way to geometrize linear characters of groups of the form $\underline{H}(R)$, where $\underline{H}$ is a smooth group scheme over $R$. 
In order to do so in a systematic manner, we begin this paper by describing the function-sheaf dictionary for characters of arbitrary smooth group schemes over finite fields. 
When coupled with the Greenberg transform, this dictionary will allow for the geometrization of certain linear characters of $\underline{H}(R)$.

The function-sheaf dictionary over a finite field $k$ \cite{deligne:SGA4.5}*{Sommes trig.}
provides a way of encoding functions on the $k$-rational points of an algebraic group $G$
as $\ell$-adic local systems on $G$.  More specifically, if $G$ is a connected, commutative, algebraic group
then there is a certain category $\CS(G)$ of rank-one local systems on $G$ and an
explicit isomorphism between isomorphism classes
of objects in $\CS(G)$ and $G(k)^* \ceq \Hom(G(k), \EEx)$; 
the isomorphism is given by mapping $\cs{L}$ to the function
$\TrFrob{G} : g \mapsto \Tr(\Frob{} \vert \cs{L}_g)$.

In previous work \cite{cunningham-roe:13a}, we studied the function-sheaf dictionary for characters
smooth commutative group schemes $G$, allowing for non-connected groups.
We gave a description of the category $\CS(G)$ in this context, as well
as an epimorphism $\TrFrob{G} : \CSiso{G} \to G(k)^*$.
In contrast to the connected case, $\TrFrob{G}$ may have nontrivial kernel;
we gave an explicit description of its kernel as $\Hh^2(\pi_0(\bG), \EEx)^{\Frob{}}$ \cite{cunningham-roe:13a}*{Theorem 3.6}.  

In this paper we repair this defect in the function-sheaf dictionary
 by describing a full subcategory $\CCS(G)$ of $\CS(G)$ so that $\TrFrob{G}$ restricts to an isomorphism $\CCSiso{G} \to G(k)^*$.
We refer to objects of $\CS(G)$ as character sheaves and objects in $\CCS(G)$ as \emph{commutative character sheaves}, since the passage from $\CS(G)$ to $\CCS(G)$ involves a condition that exchanges the inputs to the multiplication morphism on $G$ (see Definition \ref{def:CCScom}).  
When $G$ is connected, all character sheaves on $G$ are commutative.

Category $\CCS(G)$ clarifies several questions about $\CS(G)$. 
Invisible character sheaves \cite{cunningham-roe:13a}*{Def. 2.8} are precisely those $\cs{L}$ with $\TrFrob{G}(\cs{L}) = 1$ that are not commutative.  Moreover, $\TrFrob{G}^{-1} : G(k)^* \to \CCSiso{G}$ provides a canonical splitting of $\TrFrob{G} : \CSiso{G} \to G(k)^*$ \cite{cunningham-roe:13a}*{Rem. 3.7}.

Next, we broaden our scope further to encompass smooth group schemes $G$ over $\Fq$ that are not necessarily commutative.
We assume $G$ is smooth, but not that it is connected, reductive or commutative. 
The category $\CS(G)$ has a straightforward generalization to this case.
We then define category $\CCS(G)$ for such $G$ and a forgetful functor to $\CS(G)$ so that $\TrFrob{G} : \CCSiso{G} \to G\ab(k)^*$
is an isomorphism.  
Since $G\ab(k)^*$ maps onto $G(k)^*$ with cokernel $G\der(k)^*$, it follows that for each character $\chi \in G(k)^*$ trivial on $G\der(k)$ there is a commutative
character sheaf $\cs{L}$ on $G$ with $\TrFrob{G}(\cs{L}) = \chi$. Moreover, we find that pullback along the quotient $q : G \to G\ab$
defines an equivalence of categories $\CCS(G\ab) \to \CCS(G)$.  
The functor $\CCS(G) \to \CS(G)$ is not
essentially surjective, missing the kinds of linear character sheaves highlighted by Kamgarpour in \cite{kamgarpour:09a}*{(1.1)}.

In order to provide further justification for referring to objects in $\CCS(G)$ as commutative character sheaves, suppose for the moment that $G$ is a connected, reductive algebraic group over $\Fq$.
Let $\gcs{L}$ be the geometric part of an object in $\CCS(G)$; see Section~\ref{sec:defs}.
Let $T$ be a maximal torus in $\bG$ and let $\gcs{L}_T$ be the restriction of $\gcs{L}$ to $T$.
Then the perverse sheaf $\gcs{L}[\dim G]$ appears in the semisimple complex $\operatorname{ind}_{B,T}^{\bG}(\gcs{L}_T)$ produced by parabolic induction.
It follows that every object in $\CCS(G)$ determines a Frobenius-stable character sheaf on $G$, in the sense of \cite{lusztig:85a}*{Def.~2.10}. 
Of course, the sheaves arising in this way represent a small part of Lusztig's geometrization of characters of representations of connected, reductive groups over finite fields, but they are precisely those needed to describe one-dimensional characters of such groups. 

Armed with the function-sheaf dictionary for smooth group schemes over finite fields, we return to the task of geometrizing Yu type data. 
The proof of Theorem~\ref{thm:geotypes} requires: Yu's work on smooth integral models \cite{Yu:models}; the geometrization of the character of the Heisenberg-Weil representation over finite fields by Gurevich \& Hadani \cite{gurevich-hadani:07a}; Lusztig's character sheaves on reductive groups over finite fields; and finally, the function-sheaf dictionary for characters of smooth group schemes over finite fields, now at our disposal in Theorem~\ref{thm:geo}.
In order to use these tools we must restrict our attention to Yu type data that satisfy two technical conditions, appearing in Section~\ref{ssec:geotypes} as Hypotheses~\ref{H0} and \ref{H1}. 
Granting these hypotheses, these pieces are assembled in Section~\ref{ssec:geotypes}, where we prove Theorem~\ref{thm:geotypes}.
With this theorem, we provide all of the ingredients needed to geometrize a class of supercuspidal representations of arbitrary depth. 
 
\bigskip

We now summarize the sections of the paper in more detail.
In Section \ref{sec:defs}, we recall the category $\CS(G)$ from \cite{cunningham-roe:13a} and note that it still makes sense when $G$ is not commutative.  
We focus on the case of commutative $G$ in Section \ref{sec:comcom},
giving the definition of a commutative character sheaf and proving our first main theorem, that
$\TrFrob{G} : \CSiso{G} \to G(k)^*$ induces an isomorphism on $\CCSiso{G}$.
Passing to the case that $G$ is non-commutative, we give the definition of and main results about commutative character sheaves in Section \ref{sec:noncom}.  
We note that we should only consider character sheaves that arise via pullback from $G\ab$ in order to eliminate those that have nontrivial restriction to the derived subgroup.  
This observation underlies the definition of commutative character sheaves for non-commutative $G$.  
We state our second main result, Theorem~\ref{thm:geo}, that pullback along the abelianization map defines an equivalence of categories $\CCS(G) \to \CCS(G\ab)$.
In Section \ref{ssec:obmor}, we use Galois cohomology to describe the relationship between $G(k)^*$ and $G\ab(k)^*$.  
We also compute the automorphism groups in $\CCS(G)$.
In Section~\ref{sec:types} we use Theorem~\ref{thm:geo} to geometrize types for certain supercuspidal representations of $p$-adic groups, in a sense made precise in Theorem~\ref{thm:geotypes}.
As preparation for the proof, we review some facts about the Heisenberg-Weil representation and its geometrization, in Section~\ref{ssec:Jacobi}.
Then, in Section~\ref{ssec:review}, we review Yu's theory of types and his study of smooth integral models.  
These elements are pulled together in Section~\ref{ssec:geotypes}, where the proof Theorem~\ref{thm:geotypes} is given.

\bigskip

We are extremely grateful to Loren Spice for explaining Yu's types for supercuspidal representations and Masoud Kamgarpour for helpful conversations.
We also thank Will Sawin and the anonymous referee, who pointed out errors in earlier versions of this paper and helped us clarify several points.

\section{Recollections and definitions} \label{sec:defs}

Let $G$ be a smooth group scheme over a finite field $\Fq$; that is, let $G$ be a group scheme over $\Fq$
for which the structure morphism $G \to \Spec{\Fq}$ is smooth in the sense of \cite{EGAIV4}*{Def 17.3.1}.
This implies $G \to \Spec{\Fq}$ is locally of finite type, but not that it is of finite type.
We remark that the identity component $G^0$ of $G$ is of finite type over $\Fq$, while the component group scheme
$\pi_0(G)$ of $G$ is an \'etale group scheme over $\Fq$, and both are smooth over $\Fq$.

In this paper we use a common formalism for Weil sheaves, writing $\cs{L}$ for the pair $(\gcs{L},\phi)$, where $\gcs{L}$ is an $\ell$-adic sheaf on $\bG \ceq G\otimes_{\Fq} \bFq$ and where $\phi : \Frob{}^*\gcs{L} \to \gcs{L}$ is an isomorphism of $\ell$-adic sheaves. 
We also follow convention by referring to $\cs{L}$  as a Weil sheaf on $G$, as in \cite{Deligne:Weil}*{D\'efinition 1.1.10}.
If $\cs{L}$ and $\cs{L}' \ceq (\gcs{L}', \phi')$ are Weil sheaves, we write $\alpha : \cs{L} \to \cs{L}'$ for a morphism $\alpha : \gcs{L} \to \gcs{L}$ such that 
\[
\begin{tikzcd}
\Frob{}^* \gcs{L} \arrow{d}[swap]{\phi} \arrow{r}{\Frob{}^*\alpha} &  \Frob{}^* \gcs{L} \arrow{d}{\phi'}\\
\gcs{L} \arrow{r}{\alpha} & \gcs{L}
\end{tikzcd}
\]
commutes.  
These conventions simplify notation considerably, but they were not employed in \cite{cunningham-roe:13a}.

We write $m : G \times G \to G$ for the multiplication morphism, and $G(k)^*$ for $\Hom(G(k), \EEx)$.
Define $\theta : G\times G \to G\times G$ by $\theta(g,h) = (h,g)$.

When $G$ is commutative, a \emph{character sheaf} on $G$ is a triple $(\gcs{L}, \mu, \phi)$,
where $\gcs{L}$ is a rank-one $\ell$-adic local system on $\bG$, 
$\mu : \bm^* \gcs{L} \to \LxL$ is an isomorphism
of sheaves on $\bG \times \bG$, 
and $\phi : \Frob{G}^* \gcs{L} \to \gcs{L}$ is an isomorphism of sheaves on $\bG$;
the triple $(\gcs{L}, \mu, \phi)$ is required to satisfy certain conditions \cite{cunningham-roe:13a}*{Def. 1.1}.
Write $\CS(G)$ for the category of character sheaves on $G$.

Even when $G$ is not commutative, the category $\CS(G)$, defined as in \cite{cunningham-roe:13a}*{Def. 1.1},
still makes sense.  In order to distinguish the resulting objects from
the character sheaves of Lusztig, we will refer to the former as \emph{linear character sheaves}
(to evoke the one-dimensional character sheaves of \cite{kamgarpour:09a}).

\section{Commutative character sheaves on commutative groups}\label{sec:comcom}

We consider first the case that $G$ is commutative, which we will later apply to the case of general smooth $G$.
Let $\cs{L}$ be a character sheaf on $G$.  Since $m = m \circ \theta$ in this case,
there is a canonical isomorphism $\xi : m^* \cs{L} \to \theta^* m^* \cs{L}$.
There is also an isomorphism $\vartheta : \cs{L}\boxtimes\cs{L} \to \theta^*(\cs{L}\boxtimes\cs{L})$
given on stalks by the canonical map $\gcs{L}_{g} \otimes \gcs{L}_{h} \to \gcs{L}_{h} \otimes \gcs{L}_{g}$.

\begin{definition}\label{def:CCScom}
A character sheaf $(\cs{L}, \mu)$ on a smooth commutative group scheme $G$ is \emph{commutative}
if the following diagram of Weil sheaves on $G \times G$ commutes.
  \[
  \begin{tikzcd}[row sep=30]
   m^*\cs{L} \arrow{d}[swap]{\xi}{m= m\circ\theta} \arrow{r}{\mu} & \cs{L}\boxtimes\cs{L} \arrow{d}{\vartheta}\\
   \theta^*(m^*\cs{L}) \arrow{r}{{\theta}^*\mu} &  \theta^*(\cs{L}\boxtimes\cs{L})
  \end{tikzcd}
  \]
We write $\CCS(G)$ for the full subcategory of $\CS(G)$ consisting of commutative character sheaves.
 \end{definition}

In \cite{cunningham-roe:13a}*{Theorem 3.6}, we showed that $\TrFrob{G} : \CSiso{G} \to G(k)^*$ is surjective and
explicitly computed its kernel.  In this section, we show that the corresponding map
$\TrFrob{G} : \CCSiso{G} \to G(k)^*$ for commutative character sheaves is an isomorphism.
We begin by reinterpreting Definition \ref{def:CCScom} in terms of cocycles.

Let $G$ be a commutative \'etale group scheme over $k$. For a character sheaf $\cs{L}$ on $G$, recall
\cite{cunningham-roe:13a}*{\S 2.3} that $S_G : \CSiso{G} \to \Hh^2(E_G^\bullet)$ is an isomorphism mapping
$[\cs{L}]$ to $[\alpha \oplus \beta]$, where $E_G^\bullet$ is the total space of the zeroth page
of the Hochschild-Serre spectral sequence, $\alpha \in \oK^0(\Weil{}, \oK^2(\bG, \EEx))$ is obtained from $\mu$ and
$\beta \in \oK^1(\Weil{}, \oK^1(\bG, \EEx))$ is obtained from $\phi$.

Let $a \in Z^2(\bG, \EEx)$ correspond to $\alpha$.  We say that $[\alpha \oplus \beta] \in \Hh^2(E_G^\bullet)$
is \emph{symmetric} if $a(x,y) = a(y,x)$ for all $x,y \in \bG$.  This condition is well defined, since every
coboundary in $B^2(\bG, \EEx)$ is symmetric.  The connection between commutative character sheaves
and symmetric classes is given in the following lemma.

\begin{lemma} \label{lem:symccslink}
Suppose $G$ is a smooth commutative group scheme, and let $\cs{L}$ be a character sheaf on $G$.
Then $\cs{L}$ is commutative if and only if $S_G(\cs{L})$ is symmetric.
\end{lemma}
\begin{proof}
The symmetry of $S_G(\cs{L})$ is a direct consequence of the commutativity of the diagram in Definition \ref{def:CCScom}
after choosing bases for each stalk.
\end{proof}

We may similarly define a symmetric class in $\Hh^2(\bG, \EEx)$ to be one represented by a symmetric $2$-cocycle.
The following lemma will allow us to show that there are no invisible commutative character sheaves.

\begin{lemma} \label{lem:symtriv}
Let $\bG$ be a commutative group.  Then the only symmetric class in $\Hh^2(\bG, \EEx)$ is the trivial class.
\end{lemma}

\begin{proof}
By the universal coefficient theorem,
\[
0 \to \Ext^1_\ZZ(\Hh_{n-1}(\bG, \ZZ), \EEx) \to \Hh^n(\bG, \EEx) \to \Hom(\Hh_n(\bG, \ZZ), \EEx) \to 0
\]
is exact for all $n > 0$.  When $n = 2$, using the fact that $\bG$ is commutative, we have that $\Hh_1(\bG, \ZZ) \cong \bG$
and that $\Hh_2(\bG, \ZZ) \cong \wedge^2 \bG$. We get
\[
0 \to \Ext^1_\ZZ(\bG, \EEx) \to \Hh^2(\bG, \EEx) \to \Hom(\wedge^2 \bG, \EEx) \to 0.
\]
The map $\Hh^2(\bG, \EEx) \to \Hom(\wedge^2 \bG, \EEx)$ maps a $2$-cocycle $f$ to the alternating function
\[
(x,y) \mapsto \frac{f(x,y)}{f(y,x)}.
\]
Thus the cohomology classes represented by symmetric cocycles are precisely those in the image of $\Ext^1_\ZZ(\bG, \EEx)$.
But $\Ext^1_\ZZ(-, \EEx)$ vanishes because $\EEx$ is divisible.
\end{proof}

\begin{lemma} \label{lem:conncomm}
If $G$ is a connected commutative algebraic group over $\Fq$ then every character sheaf on $G$ is commutative.
\end{lemma}

\begin{proof}
Suppose $S_G(\cs{L}) = [\alpha\oplus \beta]\in \Hh^2(E_G^\bullet)$.
We can use \'etale descent to see that pullback by the Lang isogeny defines an equivalence
of categories between local systems on $G$ and $G(\Fq)$-equivariant local systems on $G$.  
Thus every character sheaf $\cs{L}$ on $G$ arises through the Lang isogeny, together with a character $G(\Fq) \to \EEx$.
Pushing forward the Lang isogeny along this character defines an extension of $\bG$ by $\EEx$ whose class is fixed by Frobenius; let $a\in Z^2(\bG, \EEx)$ be a representative $2$-cocycle.
Then $a$ corresponds to the $\alpha \in \oK^0(\Weil{}, \oK^2(\bG, \EEx))$, above. 
Since the covering group of the Lang isogeny is $G(k)$, which is commutative, the class of this extension satisfies $a(x,y) = a(y,x)$ for all $x,y \in \bG$. 
This shows that $S_{G}(\cs{L})$ is symmetric.
It follows from Lemma~\ref{lem:symccslink} that $\cs{L}$ is a commutative character sheaf.
\end{proof}

\begin{theorem} \label{thm:trfrobiso}
If $G$ is a smooth commutative group scheme over $\Fq$ then $\TrFrob{G} : \CCSiso{G} \to G(\Fq)^*$ is an isomorphism.
\end{theorem}

\begin{proof}
Suppose first that $G$ is \'etale.  Consider the isomorphism of short exact sequences
\[
\begin{tikzcd}
 0 \arrow{r} & \ker \TrFrob{G} \arrow{d} \arrow{r} & \CSiso{G}\arrow{d}{S_G} \arrow{r}{\TrFrob{G}} \arrow{r} & G(\Fq)^* \arrow{d} \arrow{r} & 0\\
  0 \arrow{r} & \Hh^0(\Weil{},\Hh^2(\bG,\EEx)) \arrow{r} & \Hh^2(E^\bullet_G) \arrow{r} & \Hh^1(\Weil{},\Hh^1(\bG,\EEx)) \arrow{r} & 0
 \end{tikzcd}
 \]
from \cite{cunningham-roe:13a}*{Prop. 2.7}.

Suppose that $\cs{L}$ is a commutative character sheaf with $\trFrob{\cs{L}} = 1$, and set $[\alpha, \beta] = S_G([\cs{L}])$.
Then $S_G([\cs{L}])$ is in the image of $\Hh^2(\bG, \EEx)^\Weil{}$, so is cohomologous to
$[\alpha', 0]$.  Since $\alpha$ is symmetric and coboundaries are symmetric, $\alpha'$ is symmetric as well.
So by Lemma \ref{lem:symtriv}, $\alpha'$ is cohomologically trivial, and thus $[\cs{L}]$ is trivial as well.

To see that $\TrFrob{G}$ is still surjective on $\CCSiso{G}$, note that the character sheaf constructed in the proof of
\cite{cunningham-roe:13a}*{Prop. 2.6} has trivial $\alpha$, and is thus commutative.

For general smooth commutative group schemes, we use Lemma \ref{lem:conncomm} and the snake lemma, as in the proof of
\cite{cunningham-roe:13a}*{Theorem 3.6}
\end{proof}

\begin{remark}
Since $\Hh^0(\Weil{},\Hh^2(\bG,\EEx))$ is not necessarily trivial \cite{cunningham-roe:13a}*{Ex. 2.10}, the functor
$\CCS(G) \to \CS(G)$ is not necessarily essentially surjective.  Indeed, the invisible character sheaves \cite{cunningham-roe:13a}*{Def. 2.8}
defined in our previous paper are precisely those non-commutative character sheaves with trivial trace of Frobenius.
\end{remark}

\section{Commutative character sheaves on non-commutative groups}\label{sec:noncom}

We now consider the case of a smooth group scheme without the commutativity assumption.  Our approach
is to relate linear character sheaves on $G$ to character sheaves on its abelianization $G\ab = G / G\der$, where $G\der$ is defined by \cite{demazure:SGA3-VIB}*{D\'efinition 7.2.2} and the quotient is $G\ab$ is an instance of  \cite{demazure:SGA3-VIA}*{Section 7.2.2}.
While every character $\chi \in G(k)^*$ vanishes on $G(k)\der$, it may not vanish on $G\der(k)$.  For example,
if $k$ has odd characteristic then there are nontrivial characters $\PGL_2(k) \to \EEx$ vanishing on
$\PGL_2(k)\der = \PSL_2(k)$ (see Section \ref{sec:badchar}).  In passing to $G\ab$,
we may only hope to geometrize characters that vanish on all of $G\der(k)$.

We begin this section with the main definition in this paper - the category $\CCS(G)$ of commutative character sheaves, Definition~\ref{def:CCS}. 
This definition is delicate and somewhat technical, but it is vindicated in Theorem~\ref{thm:Gab} which shows that $\CCS(G)$ is equivalent to the category of commutative character sheaves on the abelianization $G\ab$ of $G$. 
To prove Theorem~\ref{thm:Gab} we use descent theory in Section~\ref{ssec:descent}, in the process giving insight into Definition~\ref{def:CCS}. 
Section~\ref{sec:noncom} concludes with Theorem~\ref{thm:geo}, showing that the dictionary from $\CCS(G)$
to characters of $G(k)$ in fact encompasses every character vanishing on $G\der(k)$.

\subsection{Main definition}\label{ssec:noncomdef}

Recall from Section~\ref{sec:defs} that we refer to objects in category $\CS(G)$, defined as
in \cite{cunningham-roe:13a}*{Def. 1.1}, as linear character sheaves when $G$ is smooth but
not necessarily commutative.  
We define the following category to track the trivialization on the derived subgroup;
commutative character sheaves will then be defined as a subcategory.

\begin{definition}\label{def:CSab}
Let $\CS\ab(G)$ denote the category of triples $(\cs{L},\mu,\beta)$ where $(\cs{L},\mu) \in \CS(G)$ and
$\beta : \cs{L}\vert_{G\der} \to (\EE)_{G\der}$ is an isomorphism of Weil local systems on $G\der$.
A morphism $(\cs{L},\mu,\beta)\to (\cs{L}',\mu',\beta')$ is a morphism $\alpha : (\cs{L},\mu)\to (\cs{L}',\mu')$
in $\CS(G)$ such that $\beta = \beta' \circ \alpha\vert_{G\der}$.  
\end{definition}

The reason for tracking $\beta$ is that it determines an isomorphism $\gamma : m^*\cs{L} \to \theta^*m^*\cs{L}$ which will replace the $\xi$ of Definition \ref{def:CCScom}, as follows.
Let $i : G \to G$ be inversion and $c : G\times G\to G\der$ be the commutator map, defined by $c(x,y)= xyx^{-1}y^{-1}$.
Both are morphisms of $\Fq$-schemes.
Set $m' = i \circ m \circ \theta$ and let $j\der : G\der\to G$ be inclusion; then $j\der\circ c = m \circ (m \times m')$.
Then $\beta : \cs{L}\vert_{G\der} \to (\EE)_{G\der}$ determines the isomorphism $\gamma' : m^*\cs{L} \otimes \theta^* m^* i^*\cs{L} \to (\EE)_{G\times G}$ by the diagram of isomorphisms below.
\begin{equation}
\begin{tikzcd}
\arrow[equal]{d} c^* (\cs{L}\vert_{G\der}) \arrow{r}{c^*(\beta)} 
	&  c^*((\EE)_{G\der}) \arrow[equal]{d} \\
\arrow{d}[swap]{j\der\circ c = m \circ (m \times m')} c^* j\der^* \cs{L} 
	&   (\EE)_{G\times G} \\
(m \times m')^* m^* \cs{L} \arrow{d}[swap]{(m \times m')^*(\mu)} 
	&  m^*\cs{L} \otimes \theta^* m^* i^* \cs{L}  \arrow[dashed]{u}[swap]{\gamma'} \\
(m \times m')^* (\cs{L} \boxtimes \cs{L}) \arrow[equal]{r} 
	& m^*\cs{L} \otimes (m')^*\cs{L} \arrow{u}[swap]{m' = i\circ m\circ \theta} 
\end{tikzcd}
\end{equation}
In the diagram above, the arrows labeled with equations come from canonical isomorphisms of functors on Weil sheaves derived from the equations; so, for example, the middle left isomorphism comes from $(m\times m')^* m^* \iso c^* j\der^*$ since $j\der\circ c = m \circ (m \times m')$.
Using the monoidal structure of the category of Weil local systems on $G\times G$, the isomorphism $\gamma' : m^*\cs{L} \otimes \theta^* m^* i^*\cs{L} \to (\EE)_{G\times G}$ defines an isomorphism
\[
m^*\cs{L} \to (\theta^* m^* i^*\cs{L})^\vee.
\]
Applying the canonical isomorphisms $(\theta^* m^* i^*\cs{L})^\vee \iso \theta^* m^* i^* (\cs{L}^\vee)$ and $i^*(\cs{L}^\vee) \cong \cs{L}$, this map provides the promised isomorphism
\[
\begin{tikzcd}
\gamma : m^*\cs{L} \arrow{r} & \theta^* m^* \cs{L}.
\end{tikzcd}
\]

\begin{definition}\label{def:CCS}
The category $\CCS(G)$ of commutative character sheaves on $G$ is the full subcategory of $\CS\ab(G)$ consisting of triples $(\cs{L},\mu,\beta)$ such that the following diagram of Weil sheaves on $G \times G$ commutes:
  \[
  \begin{tikzcd}[row sep=30]
   m^*\cs{L} \arrow{d}[swap]{\gamma} \arrow{r}{\mu} & \cs{L}\boxtimes\cs{L} \arrow{d}{\vartheta}\\
   \theta^*(m^*\cs{L}) \arrow{r}{{\theta}^*\mu} &  \theta^*(\cs{L}\boxtimes\cs{L}).
  \end{tikzcd}
  \]
Here $\gamma : m^*\cs{L} \to \theta^* m^* \cs{L}$ is the isomorphism built from $\beta : \cs{L}\vert_{G\der} \to (\EE)_{G\der}$ as above. 
\end{definition}

\subsection{Descent}\label{ssec:descent}

In this section we give an equivalence of categories between $\CS(G\ab)$ and $\CS\ab(G)$ and use it to describe
the pullback functor $q^* : \CS(G\ab) \to \CS(G)$ in terms of the forgetful functor $\CS\ab(G) \to \CS(G)$, where $q : G\to G\ab$ is the abelianization quotient with kernel $G\der$.
But first, in order to study commutative character sheaves, we need some auxiliary categories.

\subsubsection{Equivariant Weil local systems}\label{ssec:equivariant1}

Let $\Loc(G)$ be the category of Weil local systems on $G$.
Let $\Loc\der(G)$ be the category of $G\der$-equivariant Weil local systems on $G$, whose definition we now recall.
Let $n : G\der\times G\to G$ be the restriction of $m : G\times G\to G$ to $G\der\times G$,
let $p : G\der\times G \to G$ be projection to the second component, and let $s: G \to G\der\times G$ be given by $s(g) = (1,g)$.
Then the quotient $q : G \to G\ab$ is a regular epimorphism of smooth group schemes with kernel pair $(n,p)$.
\[
\begin{tikzcd}
G\der\times G
 \arrow[shift left=1]{r}{n}
  \arrow[shift right=1,swap]{r}{p}
&
G 
\arrow{r}{q}
& 
G\ab
\end{tikzcd}
\]
Consider the morphisms
\[
\begin{tikzcd}
G\der \times G\der \times G 
\arrow[shift left=2]{r}{b_1, b_2, b_3} 
\arrow{r}{}
\arrow[shift right=2]{r}{} 
& G\der \times G 
\arrow[shift left=1]{r}{n}
\arrow[shift right=1]{r}[swap]{p}
 & G 
\end{tikzcd}
\]
defined by 
\begin{align*}
b_1(h_1,h_2,g) &= (h_1h_2,g) \\
b_2(h_1,h_2,g) &= (h_1,h_2g) \\
b_3(h_1,h_2,g) &= (h_2,g).
\end{align*}
Note that
\begin{align*}
n\circ b_1 &= n\circ b_2 \\
n\circ b_3 &= p\circ b_2 \numberthis \label{eqn:bap}\\
p\circ b_1 &= p\circ b_3.
\end{align*}
A $G\der$-equivariant Weil local system on $G$ is a Weil local system $\cs{L}$ on $G$ together with an isomorphism  
\[
\nu : n^*\cs{L} \to p^*\cs{L}
\] 
of Weil local systems on $G\der\times G$ such that 
\begin{equation}\label{E1}
s^*(\nu) = \id_{\cs{L}}
\end{equation}
and such that the following diagram of isomorphisms of Weil local systems on $G\der\times G\der\times G$ commutes. 
\begin{equation}\label{E2}
\begin{tikzcd}
\ &  \arrow{dl}{n\circ b_1 = n\circ b_2}  b_2^*  n^*\, \mathcal{L} \arrow{rr}{b_2^*(\nu)} && b_2^* p^*\, \mathcal{L} \arrow{dr}[swap]{p\circ b_2 = n\circ b_3} & \\
b_1^* n^*\, \mathcal{L} \arrow{dr}{b_1^*(\nu)} &&&&   \arrow{dl}[swap]{b_3^*(\nu)} b_3^* n^*\, \mathcal{L} \\
& b_1^* p^*\, \mathcal{L}  && \arrow{ll}[swap]{p\circ b_3 = p\circ b_1} b_3^* p^*\, \mathcal{L} & 
\end{tikzcd}
\end{equation}
Morphisms of $H$-equivariant Weil local systems $(\cs{L},\nu)\to (\cs{L}',\nu')$ are morphisms of Weil local systems $\alpha: \cs{L}\to \cs{L}'$ for which the diagram
\begin{equation}\label{E3}
\begin{tikzcd}
\arrow{d}[swap]{\nu} n^*\cs{L} \arrow{r}{n^*(\alpha)} & n^*\cs{L}' \arrow{d}{\nu'} \\
p^*\cs{L} \arrow{r}{p^*(\alpha)} & p^*\cs{L}'
\end{tikzcd}
\end{equation}
commutes.
This defines $\Loc\der(G)$, the category of $G\der$-equivariant Weil local systems on $G$.
The reader will recognize this notion as the Weil local system version of equivariant sheaves for the action $n$ of $G\der$ on $G$, as can be found, for example, in \cite{bernstein-lunts:equivariant}*{0.2}. 

Let $\Loc(G\ab)$ be the category of Weil local systems on $G\ab$.
If $\cs{L}\ab\in \Loc(G\ab)$ then $q^*\cs{L}\ab\in \Loc(G)$ comes equipped with a canonical isomorphism $\nu(\cs{L}\ab) : n^* \cs{L} \to p^* \cs{L}$ defined by the following diagram of isomorphisms.
\[
\begin{tikzcd}[column sep=40]
n^*\cs{L} \arrow[equal]{d} \arrow[dashed]{r}{\nu(\cs{L}\ab)} &  \arrow[equal]{d} p^*\cs{L} \\
n^* ( q^*\cs{L}\ab) \arrow[equal]{r}{q\circ n = q\circ p}
& p^* (q^*\cs{L}\ab)
\end{tikzcd}
\]
Then $(q^*\cs{L}\ab,\nu(\cs{L}\ab))$ satisfies \eqref{E1} and \eqref{E2}, so $(q^*\cs{L}\ab,\nu(\cs{L}\ab)) \in \Loc\der(G)$.
Moreover, if $\alpha\ab : \cs{L}\ab \to \cs{L}\ab$ is a morphism in $\Loc(G\ab)$ then $q^*(\alpha\ab)$ satisfies the condition in \eqref{E3}, so $q^*(\alpha\ab)$ is a morphism in $\Loc\der(G)$.
This defines the functor
\[
L : \Loc(G\ab)\to \Loc\der(G)
\]

\begin{lemma}\label{lemma:descent}
Suppose $G$ is a smooth group scheme.  
The functor $L : \Loc(G\ab)\to \Loc\der(G)$ is an equivalence.
\end{lemma}
\begin{proof}
The quotient $q : G \to G\ab$ is an $G\der$-torsor in the fppf topology by \cite{demazure:SGA3-VIA}*{Thm. 3.2}, and thus a $G\der$-torsor in the fpqc topology.
The lemma is now a result from descent theory, arguing as in \cite{Vistoli:notes}*{Theorem 4.46} for example. 
\end{proof}

\subsubsection{Equivariant linear character sheaves}\label{ssec:equivariant2}

With reference to Section~\ref{ssec:equivariant1}, we define a $G\der$-equivariant linear character sheaf on $G$ to be a triple $(\cs{L},\mu, \nu)$, where $(\cs{L},\mu)$ is a linear character sheaf and $(\cs{L},\nu)$ is an $G\der$-equivariant Weil local system.  
A morphism of $G\der$-equivariant linear character sheaves $(\cs{L},\mu,\nu) \to (\cs{L}',\mu',\nu')$ is a morphism of $G\der$-equivariant Weil local systems $\alpha : \cs{L}\to \cs{L}'$ which is also a morphism of linear character sheaves; 
Let $\CS\der(G)$ be the category of $G\der$-equivariant linear character sheaves on $G$.

Consider the functor
\[
q^* : \CS(G\ab) \to \CS(G)
\]
given on objects by $(\cs{L}\ab,\mu\ab) \mapsto (q^*\cs{L}\ab, (q^2)^* \mu\ab)$; this is an instance of \cite{cunningham-roe:13a}*{Lem. 1.4}.
To see that $(q^*\cs{L}\ab, (q^2)^* \mu\ab)$ is indeed a linear character sheaf on $G$,
verify \cite{cunningham-roe:13a}*{CS.3}.
Now set $L(\cs{L}\ab) = (\cs{L},\nu)$, where $L : \Loc(G\ab) \to \Loc\der(G)$ is the comparison functor above, so $\cs{L} = q^*\cs{L}\ab$ and $\nu = \nu(\cs{L}\ab)$.
Then $(\cs{L}, \mu,\nu)$ is an object in $\CS\der(G)$.
If $\alpha\ab : (\cs{L}\ab,\mu\ab) \to (\cs{L}\ab',\mu\ab')$ is a morphism in $\CS(G\ab)$, 
then $q^*(\alpha\ab) : (\cs{L},\mu) \to (\cs{L}',\mu')$ satisfies \cite{cunningham-roe:13a}*{CS4}, so $\alpha = q^*(\alpha\ab)$ is a morphism in $\CS(G)$.
These observations define the comparison functor
\[
q\ab^* : \CS(G\ab) \to \CS\der(G)
\]
and also show that the functor $q^* : \CS(G\ab) \to \CS(G)$ factors according to the following commuting diagram of functors
\begin{equation}\label{qH}
\begin{tikzcd}
\CS(G) &\arrow{l}[swap]{q^*} 
\CS(G\ab) \arrow{dl}{q\ab^*}\\
\arrow{u}{\text{forget}} \CS\der(G). & 
\end{tikzcd}
\end{equation}
The definition of $q\ab^* : \CS(G\ab)\to \CS\der(G)$ will be revisited in the proof of Proposition~\ref{prop:CSGabab}.

Set $G^2 = G\times G$, so $G^2\der = G\der\times G\der$ and $G^2\ab = G\ab\times G\ab$.
Likewise define $n^2 : G\der^2 \times G^2 \to G^2$ and $p^2 : G\der^2\times G \to G$.

\begin{lemma}\label{lem:HH}
If $(\cs{L},\mu,\nu)$ is a $G\der$-equivariant linear character sheaf on $G$ then $\mu : m^*\cs{L} \to \cs{L}\boxtimes \cs{L}$ 
is a morphism of $G^2\der$-equivariant Weil local systems on $G^2$, as defined in Section \ref{ssec:equivariant1}.
\end{lemma}

\begin{proof} 
Define
\begin{align*}
d : G\der\times G\der\times G\times G &\to G\der\times G\times G\der\times G \\
(h_1,h_2,g_1,g_2) &\mapsto (h_1, g_1, h_2, g_2)\\
n_2 : G\der\times G\times G\der\times G &\to G\times G \\
(h_1,g_1,h_2,g_2) &\mapsto (h_1g_1, h_2g_2) \\
p_2 : G\der\times G\times G\der\times G &\to G\times G \\
(h_1,g_1,h_2,g_2) &\mapsto ( g_1,g_2).
\end{align*}
The following diagram defines the isomorphisms needed to see that both $m^*\cs{L}$ and $\cs{L}\boxtimes\cs{L}$ are $G\der^2$-equivariant Weil local systems.
\[
\begin{tikzcd}[column sep=40]
\arrow{d}[swap]{n_2^*(\mu)} n_2^* (m^*\cs{L}) \arrow[dashed]{r} 
	& \arrow{d}{p_2^*(\mu)} p_2^*(m^*\cs{L})\\
\arrow{d}[swap]{n_2= n^2\circ d} n_2^*(\cs{L}\boxtimes\cs{L}) \arrow[dashed]{r} 
	& \arrow{d}{p_2= p^2\circ d} p_2^*(\cs{L}\boxtimes\cs{L}) \\
d^* (n^*\cs{L}\boxtimes n^*\cs{L}) \arrow{r}{d^*(\nu\boxtimes \nu)}& d^*(p^*\cs{L}\boxtimes p^*\cs{L}) 
\end{tikzcd}
\]
The dashed arrows both satisfy \eqref{E1} and \eqref{E2} as they apply here.
This diagram also shows that $\mu : m^*\cs{L} \to \cs{L}\boxtimes \cs{L}$ is a morphism of $G\der^2$-equivariant local systems, since it satisfies \eqref{E3} as it applies here.
\end{proof}

\begin{proposition} \label{prop:CSGabab}
Suppose $G$ is a smooth group scheme.  
Then pullback along $q : G \to G\ab$ defines an equivalence $\CS(G\ab) \to \CS\der(G)$.
\end{proposition}
\begin{proof}
Let $L^2 : \Loc(G^2\ab) \to \Loc\der(G^2)$ be the comparison functor for the quotient $q^2  : G^2 \to G^2\ab$. 
Then $L^2$ is also an equivalence by Lemma~\ref{lemma:descent}.
Moreover, using Lemma~\ref{lem:HH}, we may rewrite the comparison functor $q\ab^*$ on objects by
\[
\begin{array}{rcl}
\CS(G\ab) &\to& \CS\der(G) \\
(\cs{L}\ab,\mu\ab) &\mapsto& (L(\cs{L}\ab), L^2(\mu\ab)).
\end{array}
\]
and on morphisms by $\alpha \mapsto L(\alpha)$.
The proposition now follows from the fact that both $L$ and $L^2$ are equivalences.
\end{proof}

\subsubsection{Rigidification} 

We may now relate $\CS\ab(G)$ to $\CS(G\ab)$.

\begin{proposition}\label{prop:nubeta}
The categories $\CS\ab(G)$ and $\CS(G\ab)$ are equivalent.
\end{proposition}

\begin{proof}
In light of Proposition~\ref{prop:CSGabab}, it suffices to exhibit an equivalence  between $\CS\der(G)$ and $\CS\ab(G)$.

We begin by defining a functor $\CS\ab(G) \to \CS\der(G)$.
Let $j\der : G\der \to G$ be the kernel of $q:G \to G\ab$ 
Define $k: G\der\times G\to G\times G$ by $k(h,g) = (j\der(h),g)$.
Then for $(\cs{L},\mu,\nu) \in \CS\der(G)$, define $\beta : \cs{L}\vert_{G\der} \to (\EE)_{G\der}$ by the following diagram.
\begin{equation}\label{eqn:betanu}
\begin{tikzcd}
& n^* \cs{L} \arrow[equal]{dl}{m\circ k = n} \arrow{r}{\nu} & p^*\cs{L} &  \\ 
k^* m^* \cs{L} \arrow{dr}{k^*(\mu)} &&& (\EE)_{G\der} \boxtimes \cs{L} \arrow[equal]{ul} \\
& k^*(\cs{L}\boxtimes\cs{L}) \arrow[equal]{r} & \cs{L}\vert_{G\der}\boxtimes \cs{L} \arrow{ur}{\beta\boxtimes \id_\cs{L}}
\end{tikzcd}
\end{equation}
Then $\beta = r\der \circ j^*(\nu)$, where $r\der : (\cs{L}_1)_{G\der} \to (\EE)_{G\der}$ is the isomorphism of constant local systems  determined by the rigidification $r: \cs{L}_1 \to \EE$ of $\cs{L}$ determined by $\mu$ (see \cite{cunningham-roe:13a}*{Remark~1.11}) so that
\begin{equation}\label{eqn:rigidification}
\begin{tikzcd}
\cs{L}_1 \arrow{r}{r} \arrow{d}{\mu_{(1,1)}}  & \EE \arrow[equal]{d} \\
\cs{L}_1\otimes\cs{L}_1 \arrow{r}{r\otimes r}  & \EE\otimes  \EE
\end{tikzcd}
\end{equation}
commutes where, as usual, $=$ denotes a canonical isomorphism.
To show that $(\cs{L},\mu,\beta) \in \CS\ab(G)$, we must see that $\beta$ is a morphism in $\CS(G\der)$ by showing that the following diagram commutes in $\Loc(G\der^2)$.
\begin{equation}\label{eqn:betaGder}
\begin{tikzcd}
m\der^*\cs{L}\vert_{G\der} \arrow{rr}{m\der^* \beta} \arrow[equal]{d} && \arrow[equal]{d} (\EE)_{G\der^2}\\
\cs{L}\vert_{G\der} \boxtimes \cs{L}\vert_{G\der} \arrow{rr}{\beta\boxtimes\beta} &&  (\EE)_{G\der}\boxtimes (\EE)_{G\der} 
\end{tikzcd}
\end{equation}
By Proposition~\ref{prop:CSGabab}, $\cs{L} = q^*\cs{L}\ab$, for some $(\cs{L}\ab,\mu\ab)\in \CS(G\ab)$.
Thus, $\cs{L}\vert_{G\der} = (\cs{L}_1)_{G\der}$ and the diagram above becomes a diagram of constant Weil local systems on $G\der^2$.
We can therefore test  whether  diagram commutes by evaluating all local systems at $(1,1)\in G\der^2$.  Doing this recovers \eqref{eqn:rigidification} from \eqref{eqn:betaGder}, so \eqref{eqn:rigidification} commutes.
After confirming that morphisms that commute with $\mu$ and $\beta$ also commute with $\mu$ and $\nu$, this concludes the definition of the functor  $\CS\ab(G) \to \CS\der(G)$.

The adjoint functor $\CS\der(G) \to \CS\ab(G)$ is given by a  similar  strategy: given $(\cs{L},\mu,\beta)\in \CS\der(G)$, define $(\cs{L},\mu,\nu)$ again using  \eqref{eqn:betanu}.
Then verify that $(\cs{L},\nu)$ satisfies \eqref{E2} as it applies here, so that $(\cs{L},\nu)\in \Loc\der(G)$.
After confirming that morphisms that commute with $\mu$ and $\nu$ also commute with $\mu$ and $\beta$, this concludes the definition of the functor  $\CS\der(G) \to \CS\ab(G)$.
From these constructions,  it is now clear that $\CS\der(G) \to \CS\ab(G)$ is an equivalence.
\end{proof}

\begin{corollary} 
If $G$ is a smooth group scheme and $(\cs{L},\mu) \in \CS(G)$, then
the restriction of $(\cs{L},\mu)$ to $G\der$ is trivial if and only if $(\cs{L},\mu) \cong q^*(\cs{L}\ab,\mu\ab)$ in $\CS(G)$, for some $(\cs{L}\ab,\mu\ab) \in \CS(G\ab)$.
\end{corollary}

\begin{proof}
Notation as in the proof of Proposition~\ref{prop:CSGabab}.
Consider the following diagram.
\[
\begin{tikzcd}
\CS(G\der) & \arrow{l}[swap]{j\der^*} \CS(G) &\arrow{l}[swap]{q^*} 
\CS(G\ab) \arrow{dl}{q\ab^*} \\
& \CS\ab(G) \arrow{u}{\text{forget}}  & 
\end{tikzcd}
\]

Now, suppose $(\cs{L},\mu)\in \CS(G)$ and there is an isomorphism $\beta : \cs{L}\vert_{G\der} \to (\EE)_{G\der}$ so that
$(\cs{L},\mu,\beta) \in \CS\ab(G)$.
By Proposition \ref{prop:CSGabab}, there is some $(\cs{L}\ab,\mu\ab)\in \CS(G\ab)$ with $(\cs{L},\mu,\beta) \iso q\ab^*(\cs{L}\ab,\mu\ab)$.
Applying the forgetful functor $\CS_{G\der}(G)\to \CS(G)$ to this isomorphism, it follows that $(\cs{L},\mu) \iso q^*(\cs{L}\ab,\mu\ab)$ in $\CS(G)$, as desired. 

Conversely, suppose $(\cs{L},\mu)\in \CS(G)$ and $(\cs{L},\mu) \iso q^*(\cs{L}\ab,\mu\ab)$ in $\CS(G)$.
Then 
\[
j\der^*(\cs{L},\mu) \iso j\der^*q^*(\cs{L}\ab,\mu\ab)
\]
 in $\CS(G\der)$.
Since $q\circ j\der = 1$, it follows that the restriction of $(\cs{L},\mu)$ to $G\der$ is trivial.
\end{proof}

We may interpret this corollary as measuring how far $q^*$ is from being essentially surjective.  The next result shows that it is also not full.
Let $C$ denote the cokernel of the natural map
\[
\Hom(\pi_0(\bar{G})_{\Frob{}}, \EEx) \to \Hom(\pi_0(\bar{G}\der)_{\Frob{}}, \EEx),
\]
where $\pi_0(\bar{G})_{\Frob{}}$ denotes the coinvariants of the action of Frobenius on
the component group of $\bar{G}$

\begin{corollary}
If $G$ is a smooth group scheme and $(\cs{L},\mu)$ is a linear character sheaf on $G$ with trivial restriction to $G\der$,
then the set of isomorphism classes of objects in $\CS(G\ab)$ mapping to $(\cs{L},\mu)$ under $q^*$ is a principal homogeneous space
for $C$.
\end{corollary}
\begin{proof}
By Proposition~\ref{prop:CSGabab}, it suffices to find the set of isomorphism classes in $\CS\ab(G)$ mapping to
$(\cs{L},\mu)$ under the forgetful functor.  By the previous corollary this set is nonempty.
If $(\cs{L},\mu,\beta)$ and $(\cs{L},\mu,\beta')$ both map to $(\cs{L},\mu)$, then $\beta' \circ \beta^{-1}$ is an
automorphism of the constant sheaf on $G\der$.  Conversely, if $\varphi$ is an automorphism of $(\EE)_{G\der}$
and $(\cs{L},\mu,\beta) \in \CS\ab(G)$ then $(\cs{L},\mu,\varphi \circ \beta) \in \CS\ab(G)$.  As in \cite{cunningham-roe:13a}*{Theorem 3.9},
the automorphism group is isomorphic to $\Hom(\pi_0(\bar{G}\der)_{\Frob{}}, \EEx)$.  Finally, we note that any
automorphism $\alpha$ of $(\cs{L},\mu) \in \CS(G)$ defines an isomorphism $(\cs{L},\mu,\beta \circ \alpha|_{G\der}) \to (\cs{L},\mu,\beta)$.
Applying the analogue of \cite{cunningham-roe:13a}*{Theorem 3.9} again yields the desired result.
\end{proof}

\subsection{Objects and maps in commutative character sheaves} \label{ssec:obmor}

We are now in a position to prove that commutative character sheaves on $G$ match perfectly with commutative character sheaves on $G\ab$.
We start with a method that will allow us to situate the diagram in Definition \ref{def:CCS} within $\CS\ab(G^2)$.

\begin{lemma}\label{lemma:morinCSab}
If $(\cs{L},\mu,\beta)\in \CS\ab(G)$ then $\mu  : m^*\cs{L} \to \cs{L}\boxtimes\cs{L}$, $\gamma : m^*\cs{L} \to  \theta^*(m^*\cs{L})$ and $\vartheta: \cs{L}\boxtimes\cs{L} \to \theta^*(\cs{L}\boxtimes \cs{L})$ are morphisms in $\CS\ab(G^2)$.
\end{lemma}
\begin{proof}
Define $m^2 : G^2\times G^2 \to G^2$ by $m^2(g_1,g_2,g_1',g_2') = (g_1g_1',g_2g_2')$. Also define $p^2_i : G^2\times G^2 \to G^2$ by $p^2_i(g_1,g_2,g_1,'g_2') = (g_i,g_i')$.
First we show that $m^*\cs{L}$ is an object in $\CS(G^2)$ by equipping it with an isomorphism $\mu_m^2: (m^2)^* (m^*\cs{L}) \to m^*\cs{L} \boxtimes m^*\cs{L}$ defined by the diagram below.
\[
\begin{tikzcd}
(m^2)^* (m^*\cs{L}) \arrow{d}{(m^2)^*\mu} \arrow[dashed]{rr}{\mu_m^2} && m^*\cs{L} \boxtimes m^*\cs{L}\\
 (m^2)^*(\cs{L}\boxtimes\cs{L}) \arrow[equal]{r} & (m^2)^*(p_1)^*\cs{L} \otimes (m^2)^*(p_2)^*\cs{L} \arrow{r} & (p^2_1)^*m^*\cs{L} \otimes (p^2_2)^*m^*\cs{L} \arrow[equal]{u}
\end{tikzcd}
\] 
The pair $(m^*\cs{L},\mu_m^2)$ satisfies the conditions appearing in \cite{cunningham-roe:13a}*{Def. 1.1}.
The restriction of $m^*\cs{L}$ to $G^2\der= G\der\times G\der$ is canonically isomorphic to $(\EE)_{G^2\der}$ by
\[
\begin{tikzcd}
(m^*\cs{L})\vert_{G^2\der} \arrow{d}{\mu\vert_{G^2\der}} \arrow[dashed]{r}{\beta^2_m} & (\EE)_{G^2\der}\\
(\cs{L}\boxtimes\cs{L})\vert_{G^2\der} \arrow[equal]{r} & (\cs{L}\vert_{G\der})\boxtimes(\cs{L}\vert_{G\der}) \arrow{u}{\beta\boxtimes\beta}.
\end{tikzcd}
\]
This shows that $(m^*\cs{L},\mu_m^2,\beta_m^2) \in \CS\ab(G^2)$. 
Similar work defines $(\cs{L}\boxtimes\cs{L},\mu_\boxtimes^2,\beta_\boxtimes^2) \in \CS\ab(G^2)$.
By construction, $\mu: m^*\cs{L} \to \cs{L}\boxtimes\cs{L}$ is a morphism in $\CS\ab(G^2)$.
Similar work shows that $\gamma : m^*\cs{L} \to  \theta^*(m^*\cs{L})$ and $\vartheta: \cs{L}\boxtimes\cs{L} \to \theta^*(\cs{L}\boxtimes \cs{L})$ are also morphisms in $\CS\ab(G^2)$.
\end{proof}

Suppose $G$ is commutative, so $G\der = 1$. 
Suppose $(\cs{L},\mu,\beta)$ is an object in $\CS\ab(G)$.
Then $\beta : \cs{L}_1\to \EE$ is an isomorphism in $\CS(1)$, which is unique by \cite{cunningham-roe:13a}*{Theorem 3.9}.
Tracing through the construction of $\gamma : m^*\cs{L} \to \theta^*m^*\cs{L}$ from $\beta : \cs{L}_1\to \EE$, we find
that $\gamma : m^*\cs{L} \to \theta^*m^*\cs{L}$ is the canonical isomorphism coming from the equation $m = m \circ \theta$. 
Thus, when $G$ is commutative, Definition~\ref{def:CCS} agrees with Definition~\ref{def:CCScom}.
The next result generalizes this observation.

\begin{theorem}\label{thm:Gab}
Pull-back along the abelianization $q : G \to G\ab$
defines an equivalence of categories
\[
\CCS(G\ab) \to \CCS(G).
\]
\end{theorem}

\begin{proof}
By definition, $\CCS(G)$ is a full subcategory of $\CS\ab(G)$; likewise, $\CCS(G\ab)$ is a full subcategory of $\CS\ab(G\ab)$.
We have just seen that $\CS\ab(G\ab)$ is equivalent to $\CS(G\ab)$.
By Proposition~\ref{prop:CSGabab}, pullback along the abelianization $q : G \to G\ab$ induces an equivalence $q^*\ab : \CS(G\ab)\to \CS\ab(G)$.
Thus, the functor $\CS\ab(G\ab)\to \CS\ab(G)$ induced by pullback along $q$ is an equivalence.
The functor $\CCS(G\ab)\to \CCS(G)$ under consideration is the restriction of $\CS\ab(G\ab)\to \CS\ab(G)$ to the subcategory $\CCS(G\ab)$. 
\[
\begin{tikzcd}
{} & \arrow{dl}[swap]{q^*\ab} \CS(G\ab)\\
\CS\ab(G) & \arrow{l} \CS\ab(G\ab) \arrow{u}[swap]{\text{equiv.}}\\
\CCS(G) \arrow[>->]{u} & \arrow{l} \CCS(G\ab) \arrow[>->]{u}
\end{tikzcd}
\]
To prove the theorem, it is now sufficient to show that $\CCS(G\ab)\to \CCS(G)$ is essentially surjective.
Suppose $(\cs{L},\nu,\beta)\in \CCS(G)$.
Then $(\cs{L},\nu,\beta)\in \CS\ab(G)$.
Let $(\cs{L}\ab,\mu\ab)\in \CS(G\ab)$ be given by the equivalences above.
Let $\xi : m\ab^*\cs{L}\ab \to \theta^* m\ab^*\cs{L}\ab$ be the isomorphism attached to $(\cs{L}\ab,\mu\ab)\in \CS(G\ab)$ as in Section~\ref{sec:comcom}.
Let $\gamma : m^*\cs{L} \to \theta^* m^*\cs{L}$ be the isomorphism attached to $\beta : \cs{L}\vert_{G\der} \to (\EE)_{G\der}$ as in Section~\ref{ssec:noncomdef}.
By Lemma~\ref{lemma:morinCSab}, the diagrams below are in $\CS(G\ab)$ (right) and $\CS\ab(G)$ (left).
\[
  \begin{tikzcd}[row sep=10, column sep = 20]
   m^*\cs{L} \arrow{dd}[swap]{\gamma} \arrow{r}{\mu} 
 & \cs{L}\boxtimes\cs{L} \arrow{dd}{\vartheta}
 &&&  m\ab^*\cs{L}\ab \arrow{dd}[swap]{\xi} \arrow{r}{\mu\ab} 
   & \cs{L}\ab\boxtimes\cs{L}\ab \arrow{dd}{\vartheta}\\
 && \  &\arrow{l}[swap]{(q^2)^*\ab}  &  &\\ 
   \theta^*(m^*\cs{L}) \arrow{r}{{\theta}^*\mu} 
&  \theta^*(\cs{L}\boxtimes\cs{L}) 
 &&&   \theta^*(m\ab^*\cs{L}\ab) \arrow{r}{{\theta}^*\mu\ab} 
 &  \theta^*(\cs{L}\ab\boxtimes\cs{L}\ab)
 \end{tikzcd}
\]
The diagram on the left is the result of applying the functor $(q^2)^*\ab$ to the one on the right; in particular $\gamma = (q^2)^*\ab \xi$.
Since $(q^2)^*\ab$ is an equivalence by Proposition~\ref{prop:CSGabab}, it follows that the diagram in Definition~\ref{def:CCS} commutes if and only if the diagram in Definition~\ref{def:CCScom} commutes. 
In other words, $(\cs{L},\mu,\beta)\in \CCS(G)$ if and only if $(\cs{L}\ab,\mu\ab)\in \CCS(G\ab)$.

\end{proof}

We may use Theorem~\ref{thm:Gab} to give a description of the morphisms and the isomorphism classes of objects in $\CCS(G)$. 

\begin{corollary}\label{cor:Gab}
The category $\CCS(G)$ is monoidal and there is a canonical isomorphism
\[
\CCSiso{G} \iso \Hom(G\ab(\Fq),\EEx).
\]
Every map in $\CCS(G)$ is either trivial or an isomorphism, and the automorphism group of any object in $\CCS(G)$ is canonically
isomorphic to $\Hom(\pi_0(\bar{G}\ab)_{\Frob{}},\EEx)$.
\end{corollary}

\begin{proof}
The first claim follows from Theorems \ref{thm:trfrobiso} and \ref{thm:Gab}. Let us write $(\cs{L},\mu,\beta) \mapsto (\cs{L}\ab,\mu\ab)$ to indicate the equivalence appearing in Theorem~\ref{thm:Gab};
then \[\Aut_{\CCS(G)}(\cs{L},\mu,\beta) = \Aut_{\CCS(G\ab)}(\cs{L}\ab,\mu\ab).\] 
By \cite{cunningham-roe:13a}*{Theorem 3.9},  $\Aut_{\CCS(G\ab)}(\cs{L}\ab,\mu\ab) = \Hom(\pi_0(\bar{G}\ab)_{\Frob{}},\EEx)$.
\end{proof}

\subsection{Geometrizing characters trivial on the derived subgroup} \label{ssec:geo}

Corollary~\ref{cor:Gab} shows that commutative character sheaves on $G$ provide a natural geometrization of characters of $G\ab(k)$. In Theorem~\ref{thm:geo} we take this one small step further by exploring the relation between characters of $G(k)$ and objects in $\CCS(G)$. 

\begin{theorem}\label{thm:geo}
The trace of Frobenius $\TrFrob{} : \CCSiso{G}\to G(\Fq)^*$ fits into the following diagram,
\[
\begin{tikzcd}
\ & & \CCSiso{G\ab} \arrow{d}{\TrFrob{}}[swap]{\iso} \arrow{r}{\iso} & \CCSiso{G} \arrow{d}{\TrFrob{}}& \\
1 \arrow{r} & \Delta_G^* \arrow{r} & G\ab(k)^* \arrow{r} & G(k)^* \arrow{r} & G\der(k)^* \arrow{r} & 1,
\end{tikzcd}
\]
where $\Delta_G$ is the image of the connecting homomorphism $G\ab(\Fq) \to \Hh^1(\Fq, G\der)$. 
Thus the category $\CCS(G)$ geometrizes characters of $G(\Fq)$ in the following sense: for every group homomorphism $\chi : G(\Fq) \to \EEx$ that vanishes on $G\der(\Fq)$, there is an object $(\cs{L},\mu,\beta)$ in $\CCS(G)$ such that $\trFrob{\cs{L}} = \chi$. 
While the geometrization of $\chi$ is not unique, the isomorphism classes of possibilities are enumerated by $\Delta_G^*$.
\end{theorem}

\begin{proof}
By the definition of $\Delta_G$, we have a short exact sequence
\[
1 \to G\der(k) \to G(k) \to G\ab(k) \to \Delta_G \to 1.
\]
Applying $\Hom(-, \EEx)$ yields the bottom row in the statement.

By Theorem~\ref{thm:Gab}, the map $\CCSiso{G\ab} \to \CCSiso{G}$ is an isomorphism.
Moreover, since both $\CCSiso{G\ab} \to \CCSiso{G}$ and $G\ab(k)^* \to G(k)^*$ are
defined by pullback along $q$, the square in the statement of the theorem commutes.
Finally, $\TrFrob{} : \CCSiso{G\ab} \to G\ab(k)^*$ is an isomorphism by Corollary \ref{cor:Gab}.
\end{proof}

\section{Geometrizing Characters Nontrivial on the Derived Subgroup} \label{sec:badchar}

If $\chi$ is a character of $G(k)$ vanishing on $G\der(k)$ then Theorem~\ref{thm:geo} shows that there is a \emph{local system} $\cs{L}$ such that $m^*\cs{L} \iso \cs{L}\boxtimes\cs{L}$ and $\TrFrob{}(\cs{L}) = \chi$.
In this section, we give methods for geometrizing characters that do not vanish on $G\der(k)$, though the results may not be local systems or conjugation-equivariant, a notion that we now define.

\subsection{Conjugation equivariant Weil sheaves}\label{ssec:equivariant}

A Weil sheaf $\cs{L}$ is said to be \emph{conjugation-equivariant} if $a^*\cs{L} \iso p_2^*\cs{L}$, where  $a : G\times G\to G$ is conjugation $a(x,y)=xyx^{-1}$ and $p_2: G\times G\to G$ is projection $p_2(x,y)=y$. 
If $(\cs{L},\mu)$ is a linear character sheaf on $G$, then $\cs{L}$ is conjugation-equivariant. Indeed,
the isomorphism $\mu : m^*\cs{L} \to \cs{L}\boxtimes\cs{L}$ determines an isomorphism $a^*\cs{L} \to p_2^*\cs{L}$ of Weil sheaves on $G\times G$ by the following diagram.
\[
\begin{tikzcd}[column sep = 30]
a^*\cs{L} \arrow[equal]{d}[swap]{a= m\circ (m\times (i\circ p_1))} \arrow[dotted]{r} & p_2^*\cs{L}  \\
(m\times (i\circ p_1))^* m^*\cs{L} \arrow{d}[swap]{(m\times (i\circ p_1))^*\mu}  & p_1^*(\cs{L} \otimes \cs{L}^\vee) \otimes p_2^*\cs{L} \arrow[equal]{u}[swap]{\cs{L}\otimes \cs{L}^\vee \iso (\EE)_{G}} \\
(m\times (i\circ p_1))^* ( \cs{L}\boxtimes\cs{L}) \arrow[equal]{d}{i\circ p_1 = p_2\circ(m\times(i\circ p_1))}[swap]{m= p_1\circ(m\times(i\circ p_1))}  & p_1^*\cs{L} \otimes p_2^*\cs{L} \otimes p_1^*\cs{L}^\vee \arrow[equal]{u} \\
m^*\cs{L}\otimes (i\circ p_1))^*\cs{L} \arrow{r}[swap]{\mu \otimes \id} & (\cs{L}\boxtimes\cs{L})\otimes (i\circ p_1)^*\cs{L} \arrow[equal]{u}
\end{tikzcd}
\]
Although every linear character sheaf on $G$ is a conjugation-equivariant Weil local system on $G$, the converse is certainly not true. 

\subsection{Lusztig's character sheaves}\label{ssec:LCS}

For connected reductive groups $G$ over $k$, Lusztig's character sheaves are simple perverse sheaves $\IC(C,\mathcal{L})$ on $G$ that come equipped with an isomorphism $a^*\, \IC(C,\mathcal{L}) \to p_2^*\, \IC(C,\mathcal{L})$ in the triangulated category $D^b_c(G\times G;\EE)$.
Consequently, when $\IC(C,\mathcal{L})$ is also a Weil sheaf complex, its trace of Frobenius function $\trFrob{\IC(C,\cs{L})} : G(k)\to \EE$ is a class function, called the characteristic function of $\IC(C,\cs{L})$.
By \cite{lusztig:86a}*{Cor. 25.7}, Weil character sheaves on $G$ determine a basis for the vector space of class functions on $G(k)$; in particular, any character of any representation of $G(k)$ can be expressed as a ${\bar \QQ}$-linear combination of characteristic functions of Weil character sheaves on $G$.
Let us use the notation $\mathsf{D}^b_{c,G}(G;\EE)$ for the category of conjugation-equivariant objects in $\mathsf{D}^b_{c}(G;\EE)$, as defined in \cite{bernstein-lunts:equivariant} for example, and $\mathsf{KD}^b_{c,G}(G;\EE)$ for the Grothendieck group of $\mathsf{D}^b_{c,G}(G;\EE)$; we set $\mathsf{K}_{\bar\QQ}\textsc{D}^b_{c,G}(G;\EE)\ceq \mathsf{KD}^b_{c,G}(G;\EE)\otimes_{\ZZ} {\bar\QQ}$.
Thus, Lusztig's result shows that for every representation $\rho$ of $G(k)$ there is some $\cs{F}\in \mathsf{K}_{\bar\QQ}\textsc{D}^b_{c,G}(G;\EE)$ such that $\trace\rho = \trFrob{\cs{F}}$.
This result is extended in the papers beginning with \cite{lusztig:disconnected1} to reductive groups $G$ over $k$ for which the geometric component group for $G$ is cyclic.

When specialized to the case of characters $\chi$ of $G(k)$, Lusztig's theory of character sheaves geometrizes $\chi$ using $\mathsf{K}_{\bar\QQ}\textsc{D}^b_{c,G}(G;\EE)$; of course, the resulting geometrization is generally not a local system.
For an example when this method can be used to geometrize a character of $G(k)$ which is not trivial on $G\der(k)$, consider $G = \SL_2$ over $\FFF$.
Then $G\der = G$, but $G(k)\der$ is the subgroup of semisimple elements,
isomorphic to the quaternion group $Q_8$ and of index $3$.
There is thus a nontrivial character $\chi : G(k) \to \mu_3 \subset \EEx$.  
Lusztig's character sheaves are constructed by perverse extension of local systems on a certain stratification of $G$.
In the case of $\SL_2$, there are five strata:
\vspace{0.1in}

\begin{center}
\begin{spacing}{1.3}
\begin{tabular}{|l|l|}
\hline
$C_3$ & regular semisimple elements \\
$C_2$ & regular unipotent elements \\
$C_2^-$ & $\{-u : \mbox{ regular unipotent $u$}\}$ \\
$\{1\}$ & $\{1\}$ \\
$\{2\}$ & $\{2\}$ \\
\hline
\end{tabular}
\end{spacing}
\end{center}


\noindent
Each Weil character sheaf on $\SL_2$ over $\FFF$ is the perverse extension of one of the following local systems: the rank-$1$ constant sheaf $\mathbf{1}_{C_3}$ (resp. $\mathbf{1}_{C_2}$, $\mathbf{1}_{C_2^-}$, $\mathbf{1}_{\{1\}}$,  $\mathbf{1}_{\{2\}}$) on $C_3$ (resp. $C_2$, $C_2^-$, $\{1\}$, $\{2\}$), or the non-trivial rank-$1$ local system $\cs{E}_{C_2}$  (resp.  $\cs{E}_{C_2^-}$) on $C_2^{-}$  (resp.  $\cs{E}_{C_2^-}$) trivialized by the double cover of $C_2$ (resp. $C_2^-$).
The following table describes these perverse sheaves by showing how they decompose in the Grothendieck group into standard sheaves on $\SL_2$, and also gives the values of the trace of Frobenius on the $7$ conjugacy classes in $\SL_2(\FFF)$.

\vspace{0.1in}

\begin{spacing}{1.3}
\[
\begin{array}{|l|l|ccccccc|}
\hline
\text{Perverse} & \text{Standard} & \multicolumn{7}{|c|}{\text{Trace of Frobenius values at conjugacy classes}} \\ \cline{3-9}
\text{Sheaf} & \text{Sheaves} & (\begin{smallmatrix}1 & 0 \\ 0 & 1\end{smallmatrix}) & (\begin{smallmatrix} 2 & 0\\ 0 & 2\end{smallmatrix}) & (\begin{smallmatrix}0 & 1 \\ 2 & 0\end{smallmatrix})  & (\begin{smallmatrix} 1 & 1\\ 0 & 1 \end{smallmatrix})  & (\begin{smallmatrix}1 & 2 \\ 0 & 1 \end{smallmatrix}) & (\begin{smallmatrix} 2 & 1\\ 0 & 2 \end{smallmatrix})  & (\begin{smallmatrix}2 & 2 \\ 0 & 2 \end{smallmatrix}) \\
\hline
\IC(\mathbf{1}_{C_3})  &  \mathbf{1}_{\SL_2}[3]  & -1 & -1 & -1 & -1 & -1 & -1 & -1 \\
 \IC(\mathbf{1}_{C_2})  &  \mathbf{1}_{C_2}[2] \oplus \mathbf{1}_{\{1\}}[2]  & 1 & 0 & 0 & 1 & 1 & 0 & 0 \\
 \IC(\cs{E}_{C_2})  &  \cs{E}_{C_2}[2]  & 0 & 0 & 0 &  \sqrt{-3}  &  -\sqrt{-3} & 0 & 0 \\
 \IC(\mathbf{1}_{C_2^-})  &  \mathbf{1}_{C_2^-}[2] \oplus \mathbf{1}_{\{2\}}[2]  & 0 & 1 & 0 & 0 & 0 & 1 & 1 \\
 \IC(\cs{E}_{C_2^-})  &  \cs{E}_{C_2^-}[2]  & 0 & 0 & 0 & 0 & 0 &   \sqrt{-3}  &   -\sqrt{-3}  \\
 \IC(\mathbf{1}_{\{1\}})  &  \mathbf{1}_{\{1\}}[0]  & 1 & 0 & 0 & 0 & 0 & 0 & 0 \\
 \IC(\mathbf{1}_{\{2\}})  &  \mathbf{1}_{\{2\}}[0]  & 0 & 1 & 0 & 0 & 0 & 0 & 0 \\
\hline
\end{array}
\]
\end{spacing}

\vspace{0.1in}
\noindent Using the trace of Frobenius values, we find that 
\[
- \IC(\mathbf{1}_{C_3}) - \frac32 \IC(\mathbf{1}_{C_2}) + \frac12 \IC(\cs{E}_{C_2}) - \frac32 \IC(\mathbf{1}_{C_2^-}) - \frac12 \IC(\cs{E}_{C_2^-}) + \frac32 \IC(\mathbf{1}_{\{1\}}) + \frac32 \IC(\mathbf{1}_{\{2\}})
\]
geometrizes $\chi$.  
Note that all seven character sheaves appear in this geometrization, even though we start just with a character; in particular, the geometrization of the character $\chi$ given by this method is not a local system.

We remark that for $p > 3$ and $G = \SL_2$, we have $G\der(k) = G(k)\der$.
In this case, therefore, any character of $G(k)$ may be geometrized by a local system using Theorem~\ref{thm:geo}, so Lusztig's character sheaves are not required.

\subsection{Pushforward along the Lang isogeny} \label{ssec:Lang_pushforward}

For connected, commutative groups, the standard method for geometrizing a character $\chi$ is to use the Lang isogeny $L : G \to G$, defined by $L(y) = y^{-1}\Frob{G}(y)$.  The pushforward $L_! (\EE)_G$ of the constant sheaf decomposes into a direct sum of rank $1$ local systems according to characters of the automorphism group $G(k)$ of the cover $L$, and the local system $\cs{L}$ associated to $\chi^{-1}$ will have $\TrFrob{}(\cs{L}) = \chi$.

If $G$ is nonabelian or disconnected, a similar strategy sometimes succeeds in geometrizing characters $\chi : G(k) \to \EEx$.
Define $\tilde{L} : G \to G$ by $\tilde{L}(y) = \Frob{G}(y)y^{-1}$.
The fibers of $L$ are right cosets of $G(k)$ and the fibers of $\tilde{L}$ are left cosets.
Suppose that
\begin{enumerate}
\labitem{(L.1)}{L.1} $G(k) \subset \image(\tilde{L}) \cap \image(L)$ (which will hold if $G$ is connected).
\end{enumerate}
Then we may associate to each $\chi$ a function $\tilde{\chi} : G(k) \to \EEx$ as folllows.
If $z \in G(k)$, write $z = \tilde{L}(y)$ and define $\tilde{\chi}(z) = \chi(L(y))$.
Note that $\tilde{\chi}$ is well defined since changing $y$ to $ya$ for $a \in G(k)$ has the
effect of conjugating $L(y)$ by $a^{-1}$, which has no effect on the value $\chi(L(y))$.
In order to obtain a local system with trace $\chi$, we must also assume that
\begin{enumerate}
\labitem{(L.2)}{L.2} $\tilde{\chi}$ is also a character.
\end{enumerate}

\begin{proposition}\label{prop:chitilde}
Let $G$ be a smooth group scheme over $k$ and assume \ref{L.1} and \ref{L.2}.  Let $\cs{L} = (L_!(\EE)_G)_{\tilde{\chi}^{-1}}$.  Then $\TrFrob{}(\cs{L}) = \chi$.
\end{proposition}
\begin{proof}
The stalk of $\cs{L}$ at $x \in G(k)$ is
\[
\cs{L}_x = \{ s : L^{-1}(x) \to \EE : s(ay) = \tilde{\chi}^{-1}(a) s(y) \mbox{ for all $a \in G(k), y \in L^{-1}(x)$}\}.
\]
For any $y \in L^{-1}(x)$, note that $\Frob{G}^{-1}(y) = yx^{-1}$ and $\tilde{L}(y) = yxy^{-1} \in G(k)$.  Frobenius acts on $s \in \cs{L}_x$ by
\[
(^{\Frob{}}s)(y) = s(\Frob{G}^{-1}(y)) = \tilde{\chi}^{-1}(yx^{-1}y^{-1})s(y) = \tilde{\chi}(\tilde{L}(y))s(y) = \chi(x)s(y),
\]
thus proving the claim.
\end{proof}

Note that the resulting sheaf is a local system, but is not necessarily conjugation-equivariant because
the Lang map is not a homomorphism when $G$ is nonabelian.

This method of geometrization also applies to the example considered above.
Let $G = \SL_2$ over $\FFF$ and let $\chi : G(k) \to \mu_3 \subset \EEx$ be a non-trivial character.
In this case, $\tilde{\chi}= \chi^{-1}$. 
Since $G$ is connected and $\tilde{\chi}$ is a character, Proposition~\ref{prop:chitilde} produces a local system $\cs{L}$ on $G$ which geometrizes $\chi$.
This $\cs{L}$ is not conjugation-equivariant.

The same method of geometrization applies to $G = \PGL_2$ when $k$ has odd characteristic.
In this case, $G\der = G$, but $G(k)\der = \PSL_2(k)$ has index $2$ in $G(k)$.
For the non-trivial character $\chi$ of $G(k)$, $\tilde{\chi} = \chi$, so again we may use Proposition~\ref{prop:chitilde} to geometrize $\chi$ as a local system. 

\section{Application to type theory for \texorpdfstring{$p$}{p}-adic groups}\label{sec:types}

We now show how to use Theorem~\ref{thm:geo} to geometrize Yu type data and how to geometrize types for supercuspidal representations of tamely ramified $p$-adic groups.

\subsection{Quasicharacters of smooth group schemes over certain Henselian traits}\label{ssec:quasicharacters}

Recall that $R$ is the ring of integers of a local field with finite residue field $k$.
The maximal idea of $R$ will be denoted by $\mathfrak{m}$. 
Let $\underline{G}$ be a smooth group scheme over $R$.
Here we shall use  \cite{bertapelle-gonzales:Greenberg} for the definition and fundamental properties of the Greenberg transform.
Let $G$ be the Greenberg transform of $\underline{G}$; then $G$ is a group scheme over $\Fq$ and there is a canonical isomorphism
\[
G(\Fq) = \underline{G}(R).
\]
Let $\varphi : \underline{G}(R) \to \EEx$ be a quasicharacter. 
By the continuity of $\varphi$ there is some $r \in \NN$ and a factorization
\[
\begin{tikzcd}
\underline{G}(R) \arrow{rr}{\varphi} \arrow{rd} && \EEx\\
& \underline{G}(R/\mathfrak{p}^{r+1}) \arrow{ru}[swap]{\varphi_r} 
\end{tikzcd}
\] 
The least such $r\in \NN$ is called the {\it level} of $\varphi$.
For any $r\in \NN$, define $R_r \ceq R/\mathfrak{p}^{r+1}$ and
set $G_r \ceq \Gr_r^{R}(\underline{G})$, the Greenberg transform of $\underline{G}\times_{\Spec{R}}\Spec{R_r}$.
Then $G_r$ is a smooth group scheme over $\Fq$ and $G_r(\Fq) = \underline{G}(R_r)$.

\begin{proposition}\label{prop:quasicharacters}
Let $\underline{G}$ be a smooth group scheme over $R$.
Let $\varphi : \underline{G}(R) \to \EEx$ be quasicharacter and let $r$ be the level of $\varphi$.
If $\varphi_r : G_r(k)\to \EEx$ is trivial on $G_{r,\operatorname{der}}(k)$ then 
there is a commutative character sheaf $\cs{L}_r\in \CCS(G_r)$ such that 
\[
\trFrob{\cs{L}_r} =  \varphi_r.
\] 
\end{proposition}

\begin{proof}
This is a direct consequence of Theorem~\ref{thm:geo}.
\end{proof}

With $\varphi$ as above, note that if $r'$ is any integer greater or equal to the level $r$ of $\varphi$, then the pullback of $\cs{L}_r$ along $G_{r'} \to G_{r}$ is a linear character sheaf on $G_r$. 
Thus, $\cs{L}_{r'}$ is conjugation-equivariant in the sense of Section~\ref{ssec:equivariant}.

Recall that the full Greenberg transform $G \ceq \Gr^{R}(\underline{G})$ is a group scheme over $\Fq$ such that $G(\Fq) = \underline{G}(R)$; it comes equipped with a morphism $G \to G_r$.
The Weil sheaf on $G$ obtained from $\cs{L}_m$ by pullback along the morphism of group schemes $G \to G_r$ is a quasicharacter sheaf on $G$, in the sense of \cite{cunningham-roe:13a}*{Def 4.2}, such that 
\[
\trFrob{\cs{L}} = \varphi.
\]

\subsection{Jacobi theory over finite fields}\label{ssec:Jacobi}

For use below, we recall some facts about the Heisenberg-Weil representation.

Let $V$ be a finite-dimensional vector space over a finite field $\Fq$ equipped with a symplectic paring $\langle\ ,\ \rangle : V\times V \to Z$, where $Z$ is a one-dimensional vector space over $\Fq$.
Let $V^\sharp$ be the Heisenberg group determined by $(Z, \langle\ ,\ \rangle)$ \cite{gurevich-hadani:07a}*{\S 1.1}.
Let $\Sp(V)$ be the symplectic group determined by the symplectic pairing $\langle\ ,\ \rangle$; this group acts on $V^\sharp$.
The group $\Sp(V)\ltimes V^\sharp$ is called the Jacobi group. 
From the construction above, it is clear that the Jacobi group may be viewed as the $\Fq$-points of an algebraic group over $\Fq$; we will refer to that algebraic group as the Jacobi group.

Let $\psi : Z \to \EEx$ be an additive character and let $\omega_\psi$ be the Heisenberg representation on $V^\sharp$ with central character $\psi$ \cite{gurevich-hadani:07a}*{\S 1.1}. 
The Heisenberg representation determines a representation $\pi_{\psi}$ of $\Sp(V)$ with the same representation space as $\omega_\psi$ and with the defining property: for each $g\in \Sp(V)$, $\pi_\psi(g)$ determines an isomorphism of representations $\omega_\psi^g \to \omega_\psi$.
Let $W_\psi = \pi_\psi \ltimes \omega_\psi$ be the Heisenberg-Weil representation of the Jacobi group $\Sp(V)\ltimes V^\sharp$ given by $\omega_\psi$ and $\pi_\psi$ \cite{gurevich-hadani:07a}*{\S 2.2}.

By \cite{gurevich-hadani:07a}*{Theorem 3.2.2.1} (see also \cite{gurevich-hadani:11a}*{Theorem 4.5}), there is a Frobenius-stable conjugation-equivariant perverse sheaf $\cs{K}_\psi$ on $\Sp(V)\ltimes V^\sharp$  such that 
\begin{equation}\label{eqn:Jacobi}
\trFrob{\cs{K}_\psi} = \trace(W_\psi).
\end{equation}
In particular, this geometrization uses $\mathsf{D}^b_{c,\Sp(V)\ltimes V^\sharp}(\Sp(V)\ltimes V^\sharp;\EE)$, as recalled in Section~\ref{ssec:LCS}, to geometrize $\trace(W_\psi)$. 

\subsection{Review of Yu's types and associated models}\label{ssec:review}

For the rest of Section~\ref{sec:types}, $R$ is the ring of integers of a $p$-adic field.
A Yu type datum $(\oK^i,\orho^0,\varphi^i,d)$ consists of the following:
\begin{enumerate}
\labitem{Y0}{Y0} a sequence of compact groups $\oK^0 \subseteq \oK^1 \subseteq \cdots \subseteq \oK^d = \oK$;
\labitem{Y1}{Y1} a continuous representation $\orho^0$ of $\oK^0$;
\labitem{Y2}{Y2} quasicharacters $\varphi^i : \oK^i \to \CC^\times$, for $i=0, \ldots d$.
\end{enumerate}
The representation $\orho^0$ and the quasicharacters $(\varphi^0, \ldots , \varphi^d)$ enjoy certain properties which allow Yu to construct a sequence of types $(\oK^i,\orho_i)$, for $i=1, \ldots, d$.
In order to prepare for the construction of the geometric types of Theorem~\ref{thm:geotypes}  we review some further detail here.
In Table~\ref{table:notation} we explain how to convert the constructions appearing in this section into the notation of \cite{yu:01a}.

First, Yu introduces 
\begin{enumerate}
\labitem{Y3}{Y3}
compact groups $J_i\subset \oK$, for $i=0, \ldots d$, such that 
$
\oK^i = J_0\cdots J_{i}
$ 
and, for $i=0, \ldots d-1$, a natural action of $\oK^i$ on $J_{i+1}$ defining the groups $\oK^i \ltimes J_{i+1}$.
\begin{equation}\label{eq:semiprod}
\begin{tikzcd}
\ && 1 \arrow{d} && \\
\ && \oK^{i}\cap J_{i+1}\arrow{d} && \\
1 \arrow{r} & J_{i+1} \arrow{r} & \oK^i \ltimes J_{i+1} \arrow{d}{\pi_{i+1}} \arrow{r}{p_i} & \oK^i \arrow{r} & 1\\
&& \oK^{i+1} \arrow{d}{} && \\
&& 1 &&
\end{tikzcd}
\end{equation}
\end{enumerate}

Next, Yu defines a group homomorphism (in fact, a quotient) 
$
J_{i+1} \to V_{i+1}
$
where $V_{i+1}$ is a finite abelian group, the latter also given the structure of a $\Fq$-vector space.
The vector space $V_{i+1}$ is then equipped with a symplectic pairing $\langle\ ,\ \rangle_{i+1} : V_{i+1}\times V_{i+1} \to Z_{i+1}$, where $Z_{i+1}$ is a one-dimensional vector space over $\Fq$, itself equipped with an additive character $\psi_{i+1} : Z_{i+1} \to \CC^\times$.
This, in turn, is used to define a map
$
J_{i+1} \to V_{i+1}^\sharp,
$
where $V_{i+1}^\sharp$ is the Heisenberg group determined by $V_{i+1}$, $Z_{i+1}$, $\langle\ ,\ \rangle_{i+1}$ and $\psi_{i+1}$, as in Section~\ref{ssec:Jacobi}.
In fact, the quotient $J_{i+1} \to V_{i+1}^\sharp$ factors through a quotient $J_{i+1} \to H_{i+1}$ and an isomorphism $j_{i+1} : H_{i+1} \to V_{i+1}^\sharp$, where $H_{i+1}$ is a Heisenberg $p$-group in the sense of \cite{yu:01a}*{}.
Finally, Yu constructs a group homomorphism $f_{i+1} : \oK^i \to \Sp(V_{i+1})$ such that the pair $(f_{i+1}, j_{i+1})$ is a symplectic action of $\oK^i$ on $H_{i+1}$ in the sense of \cite{yu:01a}.
Taken together, this defines
\begin{enumerate}
\labitem{Y4}{Y4}  
a group homomorphism $h_{i+1} : \oK^i \ltimes J_{i+1} \to  \Sp(V_{i+1})\ltimes V_{i+1}^\sharp$ making the following diagram commute.
\[
\begin{tikzcd}
1 \arrow{r} & J_{i+1} \arrow{d} \arrow{r} & \oK^i \ltimes J_{i+1} \arrow[dashed]{d}{h_{i+1}} \arrow{r}{p_i} & \oK^i \arrow{r} \arrow{d}{f_i} & 1\\ 
1 \arrow{r} & V_{i+1}^\sharp \arrow{r} & \Sp(V_{i+1}) \ltimes V_{i+1}^\sharp \arrow{r} & \Sp(V_{i+1}) \arrow{r} & 1
\end{tikzcd}
\]
\end{enumerate}

\begin{table}[ht]
\caption{Notation conversion chart.}
\begin{spacing}{1.3}
\begin{tabular}{| c|l | l | }
\hline
\text{this paper} & \text{Jiu-Kang Yu, {\it Construction of tame}} & \cite{yu:01a} \\
 & \text{{\it supercuspidal representations}} &  \\
\hline
$\oK^0$ & $\,^\circ K^0 = G^0(F)_y$ & \cite{yu:01a}*{\S 15} \\
$\oK^{i+1}$ & $\,^\circ K^{i+1} = (\,^\circ K^0) \vec{G}^{(i+1)}(F)_{y,(0, s_0, \ldots, s_{i})}$ & \cite{yu:01a}*{\S 15} \\
$\orho^0$ & $\,^\circ \rho^0$ & \cite{yu:01a}*{\S 15} \\
$\orho_{i+1}$ & $\,^\circ \rho_{i+1}$ &  \cite{yu:01a}*{\S 15} \\
$\varphi^i$ & $\phi_i\vert_{\,^\circ K^i }$ & \cite{yu:01a}*{\S 3}\\
$J_{i+1}$ & $J^{i+1} = (G^i,G^{i+1})(F)_{y, (r_i, s_i)}$ & \cite{yu:01a}*{\S 3} \\
$V_{i+1}$ & $J^{i+1}/J^{i+1}_+ = (G^i,G^{i+1})(F)_{y, (r_i, s_i)}/ (G^i,G^{i+1})(F)_{y, (r_i, s_i^+)}$ & \cite{yu:01a}*{\S 3} \\ 
$V_{i+1}^\sharp$ & $(G^i,G^{i+1})(F)_{y, (r_i, s_i)}/ \ker(\widehat{\phi}_i\vert_{(G^i,G^{i+1})(F)_{y, (r_i, s_i^+)}})$ &  \cite{yu:01a}*{\S 4} \\
$Z_{i+1}$ & $\ker(V_{i+1}^\sharp\to V_{i+1})$ & \cite{yu:01a}*{\S 11} \\
$(f_{j+1}, j_{i+1})$ & $(f,j)$ & \cite{yu:01a}*{\S 11} \\
$\langle\ ,\ \rangle_{i+1}$ & $\langle\ ,\ \rangle$ & \cite{yu:01a}*{\S 11} \\
\hline
\end{tabular}
\end{spacing}
\label{table:notation}
\end{table}%

We can now recall how Yu uses all this to construct the types $(\oK^i,\orho_i)$; see \cite{yu:01a}*{\S\S 4, 15}.
The representations $\orho_i$ are defined recursively.
For the base case $i=0$, set $\orho_0 \ceq \orho^0\otimes \varphi^0$; see \ref{Y1} above.
Now fix $i$.
Let $W_{i+1}$ be the Heisenberg-Weil representation of the Jacobi group $\Sp(V_{i+1})\ltimes V_{i+1}^\sharp$, whose restriction to $V_{i+1}^\sharp$ has central character $\psi_{i+1}$.
Pull-back along $h_{i+1}$ to form $h_{i+1}^*(W_{i+1})$, a representation of $\oK^i \ltimes J_{i+1}$.
Write $\inf(\orho_i)$ for the representation of $\oK^i \ltimes J_{i+1}$ obtained by pulling back $\orho_i$ along $\oK^i \ltimes J_{i+1} \to \oK^i$. 
Consider the representation
\begin{equation}\label{eq:tensor}
\orho^{i+1} \ceq h_{i+1}^*(W_{i+1}) \otimes \inf(\orho_i)
\end{equation}
of $\oK^i \ltimes J_{i+1}$.
By \cite{yu:01a}*{}, the representation $\orho^{i+1}$ of $\oK^i \ltimes J_{i+1}$ is trivial on $\oK^{i}\cap J_{i+1}$ so $\orho^{i+1}$ descends to $\oK^{i+1}$. 
Set $\orho_{i+1} = \orho^{i+1}\otimes \varphi^{i+1}$.
This completes the recursive definition of the Yu $(\oK^i,\orho_i)$ for $i=0, \ldots , d$.
By \cite{Yu:models}*{Prop 10.2} there is a sequence  
\[
\underline{G}^0 \to \underline{G}^1 \to \cdots \to \underline{G}^d = \underline{G}
\]
of morphisms of affine smooth group schemes of finite type over $R$ such that, on $R$-points it gives the sequence $\oK^0 \subseteq \oK^1 \subseteq \cdots \subseteq \oK^d$ above.
Indeed, this is the main result of \cite{Yu:models}.

As explained in \cite{Yu:models}*{\S 10.4}, there is morphism of affine smooth group schemes of finite type over $R$ 
\[
\underline{J}^i \to \underline{G},
\] 
for each $i=0,\ldots d$, such that $\underline{J}^i(R) = J_i$ as a subgroup of $C$ and such that the image of the $R$-points under the multiplication map $\underline{J}^0 \times \cdots \times \underline{J}^i \to \underline{G}$ is $\oK^i$, for $i=0, \ldots , d$.
There is a natural action of $\underline{G}^i$ on $\underline{J}^{i+1}$ in the category of smooth affine group schemes over $R$ so that the group scheme
\[
\underline{G}^i \ltimes \underline{J}^{i+1}
\]
gives $(\underline{G}^i \ltimes \underline{J}^{i+1})(R) = \oK^i \ltimes J_{i+1}$

\newcommand{\reductive}{{\operatorname{red}}}

Write $\underline{J}^{i+1}_\Fq$ for the special fibre $\underline{J}^{i+1}\times_{\Spec{R}} \Spec{k}$ of $\underline{J}^{i+1}$. 
The vector space $V_{i+1}$ may realized as the $\Fq$-points on a variety $V^{i+1}$ over $\Fq$, where $V^{i+1}$, appears as a quotient $\underline{J}^{i+1}_{\Fq} \to V^{i+1}$ of algebraic groups over $\Fq$. Then the quotient $J_{i+1} \to V_{i+1}$ is realized as the composition
\[
\underline{J}^{i+1}(R) \to \underline{J}^{i+1}(\Fq) = \underline{J}^{i+1}_\Fq(\Fq) \to V^{i+1}(\Fq) = V_{i+1}.
\]
Likewise, the Heisenberg $p$-group $H_{i+1}$, appearing in \ref{ssec:review}, may be realized as a quotient of algebraic groups, and $\underline{J}^{i+1}_{\Fq} \to H^{i+1}$ as the composition 
\[
\underline{J}^{i+1}(R) \to \underline{J}^{i+1}(\Fq) = \underline{J}^{i+1}_\Fq(\Fq) \to H^{i+1}_{\Fq}(\Fq) = H_{i+1}.
\]
Finally, the group homomorphism $f_i : J_0\cdots J_i \to \Sp(V_{i+1})$ may be made geometric in much the same way. 
Writing $\underline{G}^{i}_\Fq$ for the special fibre $\underline{G}^{i}\times_{\Spec{R}} \Spec{k}$ of $\underline{G}^{i}$, and writing $\underline{G}^{i,\reductive}_\Fq$ for the reductive quotient of $\underline{G}^{i}_\Fq$, there is a quotient of algebraic groups $\underline{G}^{i,\reductive}_\Fq \to W^{i+1}_\Fq$ so that $f_i : J_0\cdots J_i \to \Sp(V_{i+1})$ is realized as the composition
\[
\underline{G}^{i}(R) \to \underline{G}^{i}(\Fq) = \underline{G}^{i}_\Fq(\Fq) \to  \underline{G}^{i,\reductive}_\Fq(\Fq) \to W^{i+1}_{\Fq}(\Fq) = \Sp(V_{i+1}).
\]

With all this, we may revisit the quotients appearing in Section~\ref{ssec:review}:
\[
\begin{tikzcd}
1 \arrow{r} & \underline{J}^{i+1}  \arrow{r} & \underline{G}^i \ltimes \underline{J}^{i+1} \arrow{r} & \underline{G}^i \arrow{r}  & 1\\
1 \arrow{r} & \underline{J}^{i+1}_\Fq \arrow{d} \arrow{u} \arrow{r} & \underline{G}_\Fq^i \ltimes \underline{J}_\Fq^{i+1} \arrow{d} \arrow{u} \arrow{r} & \underline{G}_\Fq^i \arrow{r} \arrow{d} \arrow{u} & 1\\ 
1 \arrow{r} & V_{i+1}^\sharp \arrow{r} & \Sp(V_{i+1}) \ltimes V_{i+1}^\sharp \arrow{r} & \Sp(V_{i+1}) \arrow{r} & 1,
\end{tikzcd}
\]
where the last two rows are now understood as forming a diagram in the category of algebraic groups over $\Fq$. 
This realizes the Jacobi group $\Sp(V_{i+1}) \ltimes V_{i+1}^\sharp$ as a quotient of the special fibre of the smooth group scheme $\underline{G}^i \ltimes \underline{J}^{i+1}$ over $R$. 

We may now revisit the ingredients in the construction of the representation $\rho$ of $\underline{G}(R)$ along the lines indicated by Yu and recalled in Section~\ref{ssec:review}.
\begin{enumerate}
\labitem{M0}{M0}
The compact groups $\oK^i$ appearing in \ref{Y0} have been replaced by the smooth group schemes $\underline{G}^i$.
\labitem{M1}{M1}
The continuous representation $\orho^0$ of $\oK^0$ appearing in \ref{Y1} may be interpreted as a representation of $\underline{G}^0(R)$ obtained by inflation along $\underline{G}^0(R) \to \underline{G}^0(\Fq)$ from a representation $\varrho_0$ of $\underline{G}^0(\Fq) = \underline{G}^0_\Fq(\Fq)$.
In fact, $\varrho_0$ is itself obtained by pulling back a representation $\varrho_0^\reductive$ along the $\Fq$-points of the quotient $\underline{G}^0_\Fq \to (\underline{G}^0)_\Fq^\reductive$.
\labitem{M2}{M2} The quasicharacters $\varphi^i$ appearing in \ref{Y2} are now quasicharacters of $\underline{G}^i(R)$, for $i=0, \ldots, d$.
In fact, if $r_i$ is the level of $\varphi^i$, as it appears in Section~\ref{ssec:quasicharacters}, then $\varphi^i$ is obtained by pulling back a quasicharacter $\varphi_i\ceq \varphi^i_{r_i}$ of the smooth group scheme $G_i\ceq G^i_{r_i} \ceq \Gr_{r_i}^{R}(\underline{G}^i)$ along $G^i \to G_i$.
\labitem{M3}{M3}
Diagram \eqref{eq:semiprod} in \ref{Y3} is now replaced by the following diagram of smooth group schemes over $R$.
\begin{equation}\label{eq:pimodel}
\begin{tikzcd}
\ && 1 \arrow{d} && \\
\ && \underline{G}^{i}\times_{\underline{G}} \underline{J}^{i+1}\arrow{d} && \\
1 \arrow{r} & \underline{J}^{i+1} \arrow{r} & \underline{G}^i \ltimes \underline{J}^{i+1} \arrow{d} \arrow{r} & \underline{G}^i \arrow{r} & 1\\
&& \underline{G}^{i+1} \arrow{d} && \\
&& 1 &&
\end{tikzcd}
\end{equation}
\labitem{M4}{M4}
The representation $h_{i+1}^*(W_{i+1})$ appearing in \ref{Y4} is now obtained by pulling back a representation along 
\[
(\underline{G}^i \ltimes \underline{J}^{i+1})(R) \to (\underline{G}^i \ltimes \underline{J}^{i+1})(\Fq).
\]
Let $w_{i+1}$ be that representation of $(\underline{G}^i \ltimes \underline{J}^{i+1})(\Fq) = (\underline{G}_\Fq^i \ltimes \underline{J}^{i+1}_\Fq)(\Fq)$. 
Then $w_{i+1}$ is itself obtained by pulling back the representation $W_{i+1}$ along the $\Fq$-points of the quotient
\[
\underline{G}_\Fq^i \ltimes \underline{J}^{i+1}_\Fq \to 
\Sp(V_{i+1}) \ltimes V_{i+1}^\sharp.
\]
\end{enumerate}

This brings us back to the point made in \cite{Yu:models}*{\S 10.5} as quoted in the Introduction to this paper.

\subsection{Geometrization of characters of certain types}\label{ssec:geotypes}

We may now give the main result of Section~\ref{sec:types}, Theorem~\ref{thm:geotypes}.
Since Yu's theory refers to complex representations, and since our geometrization uses $\ell$-adic sheaves, we grit our teeth and fix an isomorphism $\CC \approx \EE$.

As we recalled in Section~\ref{ssec:review}, a Yu type datum consists of compact groups $\oK^i$, a representation $\orho^0$ of $\oK^0$ and quasicharacters $\varphi^i$ of $\oK^i$, for $i=0,\ldots, d$; see \ref{Y0}, \ref{Y1} and \ref{Y2}.
In Section~\ref{ssec:review} we also saw how a Yu type datum determines smooth group schemes $\underline{G}^i$, a representation $\varrho_0\red$ of the reductive quotient $(\underline{G}^0)\red_{k}$ of the special fibre $\underline{G}^0_{k}$ of $\underline{G}^0$, and quasicharacters $\varphi_i$ of the group of $k$-rational points on the level-$r_i$ Greenberg transform $G_i$ of $\underline{G}^i$; see \ref{M0}, \ref{M1} and \ref{M2}.
Theorem~\ref{thm:geotypes} places the following conditions on the Yu type datum:
\begin{enumerate}
\labitem{H0}{H0} the geometric component group of $(\underline{G}^0)\red_{k}$ is cyclic;
\labitem{H1}{H1} for each $i=0,\ldots, d$, either
\begin{enumerate}
\labitem{H1(a)}{H1a} $G_i$ is reductive with cyclic geometric component group, or
\labitem{H1(b)}{H1b} the quasicharacter $\varphi_i$ of $G_i$ is trivial on  $G_{i,\operatorname{der}}(k)$.
\end{enumerate}
\end{enumerate}
Hypothesis~\ref{H0} will allow us to use Lusztig's character sheaves to geometrize the character of the representation of $\orho^0$.
Hypothesis~\ref{H1a} may be used to geometrize each quasicharacter $\varphi_i$ using $\mathsf{K}_{\bar\QQ}\mathsf{D}^b_{c,G_i}(G_i;\EE)$.
Alternatively, Hypothesis~\ref{H1b} allows us to use Proposition~\ref{prop:quasicharacters} to geometrize each quasicharacter $\varphi_i$ using local systems on $G_i$, or more precisely, using commutative character sheaves on $G_i$.
Using Sections~\ref{ssec:equivariant} and \ref{ssec:LCS} we see that, in both cases, the resulting geometrization may be interpreted as an element in $\mathsf{K}_{\bar\QQ}\mathsf{D}^b_{c,G_i}(G_i;\EE)$.
Note that when Hypothesis~\ref{H1b} applies, it provides a considerably simpler geometrization than when Hypothesis~\ref{H1a} applies.

We remark that if Lusztig's theory of character sheaves can be generalized to all disconnected reductive algebraic groups, then Hypothesis~\ref{H0} can be removed and Hypothesis~\ref{H1a} can be replaced with the hypothesis that each $G_i$ is reductive.

\begin{theorem}\label{thm:geotypes}
Let $(\oK^i, \orho^0, \varphi^i,d)$ be a Yu type datum that satisfies Hypotheses~\ref{H0} and \ref{H1}, and let $\orho_i$ be the representation of $\oK^i$ constructed from it in Section~\ref{ssec:review}.  Then for each $i=0,\dots, d$ there is an element $\cs{F}_i \in \mathsf{K}_{\bar\QQ}\mathsf{D}^b_{c,G^i}(G^i;\EE)$ such that
\[
\trFrob{\cs{F}_i} = \trace(\orho_i).
\]
\end{theorem}

\begin{proof}
Recall that $G^i(\Fq) = \underline{G}^i(R) = \oK^i$, canonically.
Let $r_i$ be the level of $\varphi^i$ as defined in Section~\ref{ssec:quasicharacters}.
Set $r = \max\{ r_i \tq i=0, \ldots ,d\}$.

By \cite{lusztig:disconnected1}, there is $A \in \mathsf{K}_{\bar\QQ}\mathsf{D}^b_{c,(\underline{G}^0)_\Fq\red}((\underline{G}^0)_\Fq\red;\EE)$ such that 
\[
\trFrob{A} = \trace \varrho_0^\reductive.
\]
This uses Hypothesis~\ref{H0}.
Let $A^0\in \mathsf{K}_{\bar\QQ}\mathsf{D}^b_{c,(\underline{G}^0)_\Fq}((\underline{G}^0)_\Fq;\EE)$ be obtained by pullback along the quotient $(\underline{G}^0)_\Fq \to (\underline{G}^0)_\Fq^\reductive$.
Then 
\[
\trFrob{A^0} = \trace \varrho_0.
\]
The special fibre $(\underline{G}^0)_\Fq$ of the smooth group scheme $\underline{G}^0$ is itself a smooth group scheme, and may be identified with the Greenberg transform $Q^0 = \Gr^{R}_0(\underline{G}^0)$ \cite{cunningham-roe:13a}*{\S 4.3}. 
With $r\in \NN$ as above, let ${A}_r^0$ be the equivariant Weil sheaf on the algebraic group $G_r^i$ obtained by pull-back from $A^0$ along the affine morphism $G_r^i \to Q^0$.
Factor
\begin{equation}\label{eqn:tracerho}
\begin{tikzcd}
G^0(\Fq) \arrow{rr}{\trace(\orho^0)} \arrow{rd} && \EE \\
& G_r^0(\Fq) \arrow{ru}[swap]{\trace(\orho^0)_r} 
\end{tikzcd}
\end{equation}
Observe that $\trace(\orho^0)_r$ may be recovered from ${A}_m^0$:
\[
\trFrob{{A}^0_r} = \trace(\orho^0)_r
\]

Consider the Jacobi group $\Sp(V_{i+1})\ltimes V_{i+1}^\sharp$ and the Heisenberg-Weil representation $W_{i+1}$ appearing in Section~\ref{ssec:review}.
Let $\cs{K}^{i+1}$ be the conjugation equivariant Weil sheaf on the Jacobi group, recalled in Section~\ref{ssec:Jacobi}, such that
\[
\trFrob{\cs{K}^{i+1}} = \trace(W_{i+1}).
\]
Recall from Section~\ref{ssec:review} that $\Sp(V_{i+1})\ltimes V_{i+1}^\sharp$ is a quotient of the special fibre of the smooth group scheme $\underline{G}^{i} \ltimes \underline{J}^{i+1}$.
Let $\cs{K}_0^{i+1}$ be the Weil sheaf on the special fibre of $\underline{G}^{i} \ltimes \underline{J}^{i+1}$ obtained from $W_{i+1}$ by pullback. 
Let $\cs{K}_r^{i+1}$ be the equivariant Weil sheaf on $\Gr^{R}_r(\underline{G}^{i} \ltimes \underline{J}^{i+1})$ obtained from $\cs{K}_0^{i+1}$ by pullback along the affine morphism
$\Gr^{R}_r(\underline{G}^{i} \ltimes \underline{J}^{i+1}) \to \Gr^{R}_0(\underline{G}^{i} \ltimes \underline{J}^{i+1})$.

We now define $\cs{A}^i_r  \in \mathsf{K}_{\bar\QQ}\mathsf{D}^b_{c,G^i_r}$ for $i=0,\ldots ,d$ recursively, following the construction of the representations $\orho^i$, as reviewed in Section~\ref{ssec:review}.
First, set $\cs{A}_r^0 = A_r^0$ and note that \eqref{eqn:tracerho} commutes with $\trace(\orho^0)_r$ replaced by $\trFrob{\cs{A}_r^0}$.
Using Hypothesis~\ref{H1}, let $\cs{L}_i$ be the geometrization of the quasicharacter $\varphi_i \ceq \varphi^i_{r_i}$ appearing in \ref{M2}.
If \ref{H1a} applies, then $\cs{L}_i \in \mathsf{K}_{\bar\QQ}\mathsf{D}^b_{c,G_i}(G_i;\EE)$, using \cite{lusztig:disconnected1};
if \ref{H1b} applies, then $\cs{L}_i\in \CCS(G_i)$, using Proposition~\ref{prop:quasicharacters}.
In either case, $\cs{L}_i\in \mathsf{K}_{\bar\QQ}\mathsf{D}^b_{c,G_i}(G_i;\EE)$ and
\[
\trFrob{\cs{L}_i} = \varphi_i.
\]
For each $i$, let $\cs{L}^i_{r}$ be the pull-back of $\cs{L}_i$ along $G^i_{r} \to G_i$; then $\cs{L}^i_{r}$ is a linear character sheaf and
\[
\trFrob{\cs{L}^i_{r}} = \varphi^i_{r}.
\]
Now, suppose $\cs{A}^i_r$ on $G_r^i$ is defined such that
\[
\begin{tikzcd}
G^i(\Fq) \arrow{rr}{\trace(\orho^i)} \arrow{rd} && \EE\\
& G_r^i(\Fq) \arrow{ru}[swap]{\trFrob{\cs{A}_r^i}} & 
\end{tikzcd}
\]
commutes.
Applying the Greenberg functor $\Gr^{R}_r$ to \eqref{eq:pimodel} gives
\begin{equation}\label{eq:pi}
\begin{tikzcd}
\ && 1 \arrow{d} && \\
\ && G_r^{i}\times_{G_r} J_r^{i+1} \arrow{d} && \\
1 \arrow{r} & J_r^{i+1} \arrow{r} & G_r^i \ltimes J_r^{i+1} \arrow{d}{\pi_r^{i+1}} \arrow{r}{p_r^{i}} & G_r^i \arrow{r} & 1\\
&& G_r^{i+1} \arrow{d} && \\
&& 1 && 
\end{tikzcd}
\end{equation}
where $J_r^{i+1} \ceq \Gr^{R}_r(\underline{J}^{i+1})$ and $G_r^{i} \ceq \Gr^{R}_r(\underline{G}^{i})$.
By \cite{bertapelle-gonzales:Greenberg}*{Prop 14.2}, the sequences are exact.
Consider $\cs{B}_r^{i+1}\in \mathsf{K}_{\bar\QQ}\mathsf{D}^b_{c,G_r^{i}\ltimes J_r^{i+1}}(G_r^{i}\ltimes J_r^{i+1};\EE)$ defined by
\[
\cs{B}_r^{i+1} \ceq \cs{K}_r^{i+1} \otimes (p_r^{i})^*(\cs{A}_r^{i}\otimes \cs{L}_r^{i}).
\]
Comparing with \eqref{eq:tensor}, we see that $\trFrob{\cs{B}_r^{i+1}}$ is precisely the function obtained by factoring the character of $\orho^{i+1}$ through $(\underline{G}^{i}\ltimes \underline{J}^{i+1})(R) \to (\underline{G}^{i}\ltimes \underline{J}^{i+1})(R_r)$ using the canonical identification $(G_r^{i}\times_{G_r} J_r^{i+1})(\Fq) =  (\underline{G}^{i}\ltimes \underline{J}^{i+1})(R_r)$. 
In particular, $\trFrob{\cs{B}^{i+1}_m}$ is constant on $(G_r^{i}\times_{G_r} J_r^{i+1})(\Fq)$, taking the value $\dim \orho^{i+1}$.
With reference to the morphism $\pi_r^{i+1} : G_r^i \ltimes J_r^{i+1} \to G_r^{i+1}$ from \eqref{eq:pi}, define 
\[
\cs{C}_r^{i+1} \ceq (\pi_r^{i+1})_! (\cs{B}_r^{i+1}).
\]
Then $\cs{C}_r^{i+1}\in \mathsf{K}_{\bar\QQ}\mathsf{D}^b_{c,G_r^{i+1}}(G_r^{i+1};\EE)$ and
\begin{eqnarray*}
\trFrob{\cs{C}^{i+1}_r}(x)
&=& \sum_{y\in (\pi_r^{i+1})^{-1}(x)}  \trFrob{\cs{B}_r^{i+1}}(y).
\end{eqnarray*}
Since $\trFrob{\cs{B}^{i+1}_r}$ is constant on $(G_r^{i}\times_{G_r} J_r^{i+1})(\Fq)$, it follows that 
\[
\trFrob{\cs{C}^{i+1}_r} = n \trFrob{\cs{B}^{i+1}_r}
\]
on $G_r^{i+1}(\Fq)$ with $n$ equal to the product of $\# (G_r^{i}\times_{G_r} J_r^{i+1})(\Fq)$ and $\dim \orho^{i+1}$.
Let $\cs{A}_r^{i+1}$ be the element of $\mathsf{K}_{\bar\QQ}\mathsf{D}^b_{c,G_r^{i+1}}(G_r^{i+1};\EE)$ given by $\cs{A}_r^{i+1} = \frac{1}{n} \cs{C}_r^{i+1}$. 
This completes the inductive definition of $\cs{A}_r^i$ so that the following diagram commutes.
\[
\begin{tikzcd}
G^{i+1}(\Fq) \arrow{rr}{\trace(\orho^{i+1})} \arrow{rd} && \EE\\
& G_r^{i+1}(\Fq)\arrow{ru}[swap]{{\trFrob{\cs{A}_r^{i+1}}} } & 
\end{tikzcd}
\]

Now set $\cs{F}^i_r = \cs{A}_r^{i} \otimes \cs{L}_r^i$, for $i=0, \ldots ,d$.
Then $\cs{F}^i_r \in \mathsf{K}_{\bar\QQ}\mathsf{D}^b_{c,G_r^{i}}(G_r^{i};\EE)$ such that
\[
\begin{tikzcd}
G^i(\Fq) \arrow{rr}{\trace(\orho_i)} \arrow{rd} && \EE\\
& G_r(\Fq) \arrow{ru}[swap]{\trFrob{\cs{F}^i_r}} & 
\end{tikzcd}
\]
commutes.
Define $\cs{F}^i\in \mathsf{K}_{\bar\QQ}\mathsf{D}^b_{c,G^{i}}(G^{i};\EE)$ by pulling back $\cs{F}_r^i$ along $G^i \to G^i_r$.
Then 
\[
\trFrob{\cs{F}^i} = \trace(\orho_i),
\] 
as desired.
\end{proof}

\end{document}